\documentclass{siamltex}
\usepackage{amsfonts,amsmath,amssymb,color,verbatim}
\usepackage{stmaryrd}
\usepackage{hyperref}
\usepackage[mathscr]{eucal}
\usepackage{accents}
\usepackage{pgfplots}
\usepackage{ifoddpage}
\usepackage{marginnote}
\usepackage{float}
\usepackage{placeins}
\usetikzlibrary{patterns}
\usepackage{wasysym}
\usepackage[ruled,vlined]{algorithm2e}
\usepackage[colorinlistoftodos,textwidth=3.4cm,shadow]{todonotes}
\usepackage{bm}
\usepackage{tikz}
\usetikzlibrary{arrows,snakes,backgrounds}
\usepackage{enumitem}
\usepackage{amsmath}
\usepackage{amsfonts}
\usepackage{amssymb}
\usepackage{pgf}
\usepackage{multirow}
\usepackage{caption}
\usepackage{subcaption}
\usepackage{url}
\usepackage{mathtools}
\usetikzlibrary{calc}
\usepackage{relsize}
\usepackage{mathptmx}
\usetikzlibrary{arrows,snakes,backgrounds,calc,positioning}
\usetikzlibrary{decorations.markings}
\usetikzlibrary{arrows.meta, shapes.arrows}
\tikzset{fontscale/.style = {font=\relsize{#1}}}
\newcommand{\bsp}[3]{\textbf{bsp}(#1)^{#2}_{#3}}
\newcommand{\R}[2]{\mathbb{R}^{#1\times #2}}
\newcommand{\HT}[0]{$\mathcal{H}^2$}
\newcommand{\Hm}[0]{$\mathcal{H}$}
\newcommand{\OB}[0]{\mathcal{O}}

\definecolor{royalblue}{rgb}{0.0, 0.0, 0.0}
\newtheorem{remark}{Remark}[section]
\newtheorem{example}{Example}[section]
\title{Simple non-extensive sparsification of the hierarchical matrices\footnotemark[4]}
\author{Daria A.~Sushnikova\footnotemark[5] \and Ivan V.~Oseledets\footnotemark[3]\ \footnotemark[5]}
\begin{document}

\maketitle
\renewcommand{\thefootnote}{\fnsymbol{footnote}}

\footnotetext[3]{Skolkovo Institute of Science and Technology,
Nobel St.~3, Skolkovo Innovation Center, Moscow, 143025
Moscow Region, Russia (i.oseledets@skolkovotech.ru)}
\footnotetext[5]{Institute of Numerical Mathematics Russian Academy of Sciences,
Gubkina St. 8, 119333 Moscow, Russia}
\footnotetext[4]{Section~\ref{sec:comp} is supported by Russian Foundation for Basic Research grant 17-01-00854,
Sections~\ref{sec:sp},~\ref{sec:h2tosp} are supported by Russian Foundation for Basic Research grant 16-31-60095,
Sectiom~\ref{sec:num} is supported by Russian Science Foundation grant 15-11-00033.}

\renewcommand{\thefootnote}{\arabic{footnote}}
%
%
%
%
%

\begin{abstract}

In this paper, we consider the matrices approximated in \HT~format.
The direct solution, as well as the preconditioning, of systems with such matrices is a challenging problem.
We propose a non-extensive sparse factorization of the \HT~matrix that allows to substitute direct \HT~solution with the solution of the system with an equivalent sparse matrix of the same size. The sparse factorization is constructed of parameters of the \HT~matrix. In the numerical experiments, we show the consistency of this approach in comparison to the other approximate block low-rank hierarchical solvers, such as HODLR\cite{greengard-hodlr-2016}, H2Lib\cite{Hakb-h2-lib}, and IFMM\cite{darve-ifmm_prec-2015}.

\end{abstract}


\maketitle

\section{Introduction}
Problems arising in the discretization of boundary integral equations (and a number of other problems with approximately separable kernels) lead to matrices that can be well-approximated
{\color{royalblue} by hierarchical block low-rank (\Hm\cite{hackbusch-h-1999,hackbusch-h-2000}, mosaic skeleton\cite{tee-mosaic-1996}) matrices. These are the matrices hierarchically divided into blocks, some of which has low-rank. The development of the \Hm~matrices is the \HT\cite{hackbusch-h2-2000,Borm-h2-2010} matrices, which are the hierarchical block low-rank matrices with nested bases. The nested basis property leads to the additional improvement in terms of storage and complexity of different operations such as matrix-vector products.}
Approximate solution and preconditioning of systems with \HT~matrices is {\color{royalblue}a} rapidly developed
area\cite{borm-h2lu-2013,greengard-hodlr-2016,darve-ifmm_prec-2015,Jiao-h2-ce-2017}, however, construction of the accurate, time and memory efficient factorization  that leads to approximate solution is still a challenging problem.
In this paper, we propose a new
representation of \HT~matrices. Namely, we show that \HT~factorization of matrix  $A \in \R{N}{N}$ is
equivalent to the factorization
\begin{equation}A = USV^{\top},\label{eq:usv}\end{equation}
where $S\in \R{N}{N}$ is a sparse matrix. Note that the size of matrix $S$ matches the size of
matrix $A$. $U\in \R{N}{N}$ and $V\in \R{N}{N}$ are orthogonal matrices that are
products of block-diagonal and permutation matrices.
Once the factorization~\eqref{eq:usv} is built, we can substitute a solution of the system
$$Ax = b,$$
by a solution of the system with the sparse matrix:
\begin{equation}Sy = U^{\top}b,\label{eq:sp0}\end{equation} where $x = Vy$.
The system~\eqref{eq:sp0} can be easily solved using standard sparse tools.
In this paper we propose:
\begin{itemize}
\item  Sparse \textbf{non-extensive}\footnote{The term \textbf{non-extensive} means that the sizes of the factors $S$, $U$ and $V$ are equal to the size of the matrix $A$ (in opposition to extensive\cite{Ambikasaran-ifmm-2014,darve-ifmm_prec-2015,sushnikova-se-2016} sparse factorizations of \HT~matrix).} factorization for \HT~matrix
that leads to the solver and the preconditioner.
\item The algorithm that allows to construct factors $U$, $S$ and $V$ from parameters of \HT \\matrix.
\item  Numerical comparison of the proposed method with HODLR\cite{ambikasaran-hodlersolver-2013},
IFMM\cite{darve-ifmm_prec-2015} and H2Lib\cite{Hakb-h2-lib} packages.
\end{itemize}
{\color{royalblue}
The main idea of the sparsification algorithm is the compression of the low-rank blocks (the very close idea of the compression of the fill-in blocks during the block Cholesky factorization of a sparse matrix is presented in works\cite{sushnikova-ce-2016,yang-ce-2016}).
The main difference between the presented sparse factorization and the other methods of sparsification\cite{Ambikasaran-ifmm-2014,darve-ifmm_prec-2015,sushnikova-se-2016} is a size of factors.
In the other methods, the hierarchical matrix is factorized into the sparse matrices of \textbf{larger sizes}.
The characteristic inflating coefficient is $k=5$ (it depends on the number of levels in the cluster tree\cite{Ambikasaran-ifmm-2014,sushnikova-se-2016}). The matrix extension is a major drawback since it increases the complexity of matrix computations.
We propose the factorization that takes \HT~matrix and returns the sparse factors of the \textbf{same size}.
Another drawback of the extended sparse factorizations is that resulting sparse matrix may lose a positive  definiteness of the original \HT~matrix. Proposed sparsification preserves symmetry and positive definiteness of the \HT~matrix.}

\section{Compression algorithm}
\label{sec:comp}
{\color{royalblue}Consider the dense matrix $A\in\R{N}{N}$ that can be approximated in \HT~format (has corresponding low-rank blocks).}
Formal definition of the \HT~matrix is presented in Section~\ref{sec:ht}, here
we give basic facts that are used in the current section.
Matrix $A$ is a block matrix with
following properties. It consists of two non-intersecting ``close'' and ``far''
matrices:
$$A = C_k + F_k,\quad k \in 0,\dots,L$$
block size of zero level is $B$, block size of $k$-th level $B_k$ is
$$B_k = 2^kB,$$
$C_k\in\R{N}{N}$ is  block-sparse matrix of full-rank close blocks,
and $F_k\in\R{N}{N}$ is block matrix of far blocks, see
Figure~\ref{fig:mata1}.
The matrix $F_k$ has low-rank block rows and block columns. Moreover, the nested basis
property holds: basis rows for block rows and columns on level $l$ are a subset
of basis rows at level $(l-1)$. This is used in the multilevel computations, which are described in Section~\ref{sec:lvl1}.

\begin{figure}[H]
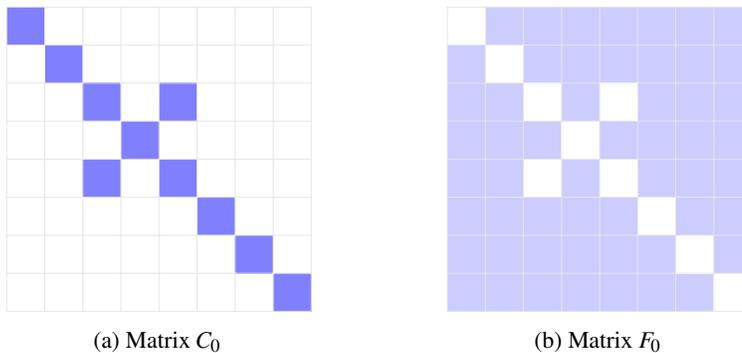

\centering
\begin{subfigure}[t]{.45\textwidth}
\centering
\resizebox{.7\textwidth}{!}{
\tikz{
   \foreach \i in {0,...,7}
   {
       \fill[blue!50!white] (\i,6-\i) rectangle (\i+1,7-\i) ;
   }
   \fill[blue!50!white] (4,4) rectangle (5,5) ;
   \fill[blue!50!white] (2,2) rectangle (3,3) ;
   \draw[step=1cm,gray!20!white] (0,-1) grid (8,7);
}}
\caption{Matrix $C_0$}
\label{fig:mata}
\end{subfigure}%
\begin{subfigure}[t]{.45\textwidth}
\centering
\resizebox{.7\textwidth}{!}{
\tikz{
   \fill[blue!20!white] (0,-1) rectangle (8,7);
   \foreach \i in {0,...,7}
   {
       \fill[white] (\i,6-\i) rectangle (\i+1,7-\i) ;
   }
   \fill[white] (4,4) rectangle (5,5) ;
   \fill[white] (2,2) rectangle (3,3) ;
   \draw[step=1cm,gray!20!white] (0,-1) grid (8,7);
}}
\caption{Matrix $F_0$}
\label{fig:mata1}
\end{subfigure}%
\caption{Close and far blocks of matrix $A$ at level $l=0$}
\label{fig:c_alg0}
\end{figure}$\\$

\subsection{Compression at zero level}
First consider the compression procedure at zero block level ($l = 0$). Assume that the number of block rows and columns at zero level is~$M$. Nonzero block $F_{ij}\in \R{B}{B}$ of far matrix $F_0$ has low rank:
$$F_{ij} \approx \widetilde{U}_i \widetilde{F}_{ij}  \widetilde{V}_j^{\top}, \quad \forall i,j \in 1,\dots, M$$
where $\widetilde{F}_{ij} \in \R{N}{N}$ is the compressed far block with the following structure:
$$ \widetilde{F}_{ij}  = \begin{bmatrix} \dot{F}_{ij} & 0\\ 0& 0\end{bmatrix},$$ where $\dot{F}_{ij} \in \R{r}{r}$. Matrices $\widetilde{U}_i \in \R{N}{N}$ and $\widetilde{V}_j \in \R{N}{N}$  are orthogonal.
The blocks in $i$-th row have the same left orthogonal
compression factor~$\widetilde{U}_i$ and all blocks in $j$-th column have the same right factor $\widetilde{V}_j^{\top}$.

The goal of the compression procedure is to sparsify the matrix $A$ by obtaining the compressed blocks $\widetilde{F}_{ij}$ instead of original blocks $F_{ij}$.
One can achieve this by finding $\widetilde{U}_i$ and $\widetilde{V}_j$ compression matrices and applying matrix $\widetilde{U}_i^{\top}$ to $i$-th row and $\widetilde{V}_j$ to $j$-th column.
We introduce the block-diagonal orthogonal
compression matrix
\begin{equation} U^{\top}_0 =
\begin{bmatrix}
\widetilde{U}_1^{\top} &0 & 0 \\
0& \ddots &0  \\
0&0 & \widetilde{U}_M^{\top}
\end{bmatrix}.
\label{eq:q}
\end{equation}
Similarly, for block columns we obtain the block-diagonal orthogonal
compression matrix \begin{equation} V_0 =
\begin{bmatrix}
\widetilde{V}_1 &0 & 0 \\
0& \ddots &0  \\
0&0 & \widetilde{V}_M
\end{bmatrix}.
\label{eq:v}
\end{equation}
Applying matrices $U_0$ and $V_0$ to the matrix $A$ we obtain the matrix $A_1$ with \emph{compressed} far matrix:
$$ A_1 = U_0^{\top} A V_0.$$
The process is illustrated in Figure~\ref{fig:c_alg01}.
\begin{figure}[H]
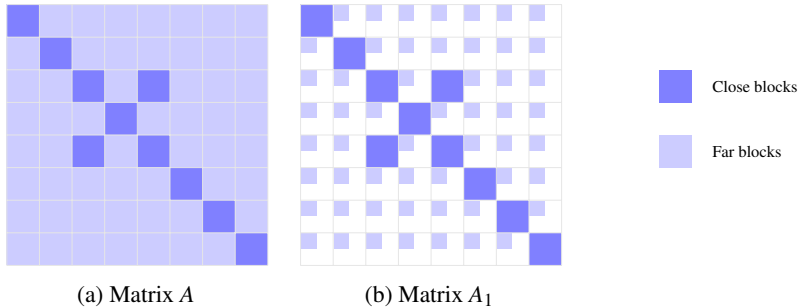

\centering
\begin{subfigure}[t]{.3\textwidth}
\centering
\resizebox{.9\textwidth}{!}{
\tikz{
   \fill[blue!20!white]  (0,-1) rectangle (8,7);
   \foreach \i in {0,...,7}
   {
       \fill[blue!50!white] (\i,6-\i) rectangle (\i+1,7-\i) ;
   }
   \fill[blue!50!white] (4,4) rectangle (5,5) ;
   \fill[blue!50!white] (2,2) rectangle (3,3) ;
   \draw[step=1cm,gray!20!white] (0,-1) grid (8,7);
}}
\caption{Matrix $A$}
\label{fig:mata}
\end{subfigure}%
\centering
\begin{subfigure}[t]{.3\textwidth}
\centering
\resizebox{.9\textwidth}{!}{
\tikz{
   \fill[blue!20!white] (0,-1) rectangle (8,7) ;
   \foreach \i in {0,...,7}
   {
       \fill[white] (\i+0.5,-1) rectangle (\i+1,7) ;
       \fill[white] (0,\i-1) rectangle (8,\i-0.5) ;
   }
   \foreach \i in {0,...,7}
   {
       \fill[blue!50!white] (\i,6-\i) rectangle (\i+1,7-\i) ;
   }
   \fill[blue!50!white] (4,4) rectangle (5,5) ;
   \fill[blue!50!white] (2,2) rectangle (3,3) ;
   \draw[step=1cm,gray!20!white] (0,-1) grid (8,7);
}}
\caption{Matrix $A_1$}
\label{fig:ahat}
\end{subfigure}%
\begin{subfigure}[t]{.3\textwidth}
\centering
\resizebox{.9\textwidth}{!}{
 \tikz{
   \draw[step=1cm,white] (-2,-5) grid (6,3);
   \fill[blue!50!white] (0,0) rectangle (1,1);
   \node [right] at (1.5,0.5) {\Large Close blocks };
   \fill[blue!20!white] (0,-1) rectangle (1,-2);
   \node [right] at (1.5,-1.5) {\Large Far blocks};
}}
\end{subfigure}%
\caption{Zero level compression}
\label{fig:c_alg01}
\end{figure}$\\$
Finally, we obtain:
$$A_1 = U_0^{\top}(C_0 + F_0)V_0 = U_0^{\top} C_0 V_0 + \widetilde{F}_1,$$
where $\widetilde{F}_1$ is a compressed far matrix which consists of blocks $\widetilde{F}_{ij}$.
Note that the matrix $C_0$ is available as one of the parameters of the \HT~format, it is a so-called ``close matrix''.

\subsection{Compression at the first level ($l=1$)}
\label{sec:lvl1}
For each block row in $A_1$ we denote the rows with zero far blocks by ``\emph{non-basis}'', and the other rows by ``\emph{first level basis}''. Assume that each block row (column) has~$r$ basis rows (columns) and $(B-r)$ non-basis.
Introduce the permutation $P_{r1}$ that puts non-basis block rows before the basis ones preserving the row order
and
permutation $P_{c1}$ that does the same for columns, see Figure~\ref{fig:a1}. For the permuted matrix
$$\widetilde{A}_1 = P_{r1}A_1P_{c1}$$
we obtain
$$\widetilde{A}_1 = \begin{bmatrix}
A_{\mathbf{n_1n_1}}&A_{\mathbf{n_1b_1}} \\
A_{\mathbf{b_1n_1}}&A_{\mathbf{b_1b_1}} \\
\end{bmatrix},$$
where $A_{\mathbf{n_1n_1}} \in \R{M(B-r)}{M(B-r)}$ is a submatrix on the intersection of non-basis rows and non-basis columns, $A_{\mathbf{b_1n_1}} \in \R{Mr}{M(B-r)}$ is on the intersection of basis rows and non-basis columns and so on, see Figure \ref{fig:a1}.
Denote the permuted far matrix:
$$\widehat{F}_1 = (P_{r1}\widetilde{F}_1P_{c1}).$$
Note that permutations $P_{r1}$ and $P_{c1}$ concentrate all nonzero blocks of compressed far zone $\widetilde{F}_1$ inside of the submatrix  $A_{\mathbf{b_1b_1}}$. 
Denote permuted close matrix:
\begin{equation}\widehat{C}_1 = (P_{r1}U_0^{\top} C_0 V_0P_{c1}).
\label{eq:cl0}
\end{equation}
Consider the submatrix  $A_{\mathbf{b_1b_1}}\in \R{Mr}{Mr}$, note that this matrix has exactly the same close and far block structure as the matrix $A$, but the block size in $A_{\mathbf{b_1b_1}}$ is $r$.
Now we join block rows and columns of the matrix $A_{\mathbf{b_1b_1}}$ by groups of $J$ blocks (e.g. $J=2$ in Figure~\ref{fig:join1} ). Assume that $Jr = B$.

We will call the grouped blocks ``big blocks''. Among these blocks, the big block that consists only of far sub-blocks will be called far,
the big block that contains at least one close small block will be referred to as close.
Denote blocks of the far matrix $\widehat{F}_1$ that become close after grouping by $\widehat{F}_{\mathbf{ml}1}$, see Figure~\ref{fig:join}.
We also introduce a new close matrix with big blocks by
\begin{equation}C_1 = \widehat{C}_1 + \widehat{F}_{\mathbf{ml}1}.
\label{eq:cl1}
\end{equation}
Denote far matrix  with big blocks by $F_1$.
Consider this joining for the block $A_{\mathbf{b_1b_1}}$:
$$A_{\mathbf{b_1b_1}} = (\widehat{C}_1)_{\mathbf{b_1b_1}} + \widehat{F}_1 = (\widehat{C}_1)_{\mathbf{b_1b_1}} + \widehat{F}_{\mathbf{ml}1} + F_1 =
(C_1)_{\mathbf{b_1b_1}} + F_1.$$
\begin{figure}[H]
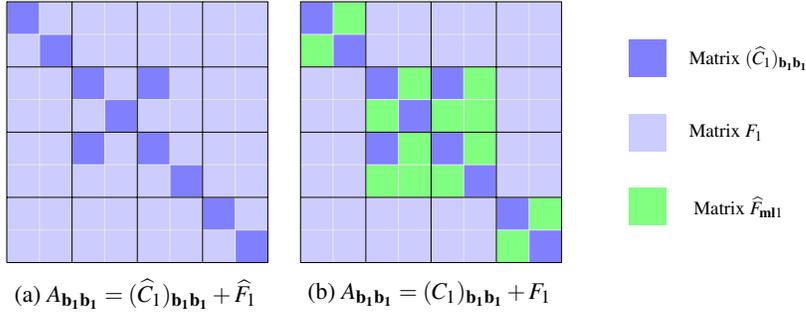

\centering
\begin{subfigure}[t]{.3\textwidth}
\centering
\resizebox{.9\textwidth}{!}{
\tikz{
   \fill[blue!20!white]  (0,-1) rectangle (8,7);
   \foreach \i in {0,...,7}
   {
       \fill[blue!50!white] (\i,6-\i) rectangle (\i+1,7-\i) ;
   }
   \fill[blue!50!white] (4,4) rectangle (5,5) ;
   \fill[blue!50!white] (2,2) rectangle (3,3) ;
   \draw[step=1cm,gray!20!white] (0,-1) grid (8,7);
   \draw[step=2cm, black, shift={(0,-3)}] (0,2) grid (8,10);
}}
\caption{$A_{\mathbf{b_1b_1}} = (\widehat{C}_1)_{\mathbf{b_1b_1}} + \widehat{F}_1 $}
\label{fig:join1}
\end{subfigure}%
\centering
\begin{subfigure}[t]{.3\textwidth}
\centering
\resizebox{.9\textwidth}{!}{
\tikz{
   \fill[blue!20!white]  (0,-1) rectangle (8,7);
   \foreach \i in {0,...,7}
   {
       \fill[blue!50!white] (\i,6-\i) rectangle (\i+1,7-\i) ;
   }
   \fill[blue!50!white] (4,4) rectangle (5,5) ;
   \fill[blue!50!white] (2,2) rectangle (3,3) ;
   \fill[green!50!white] (0,5) rectangle (1,6) ;
   \fill[green!50!white] (2,3) rectangle (3,4) ;
   \fill[green!50!white] (2,1) rectangle (3,2) ;
   \fill[green!50!white] (3,1) rectangle (4,2) ;
   \fill[green!50!white] (3,2) rectangle (4,3) ;
   \fill[green!50!white] (4,1) rectangle (5,2) ;
   \fill[green!50!white] (6,-1) rectangle (7,0) ;

   \fill[green!50!white] (7,0) rectangle (8,1) ;
   \fill[green!50!white] (5,2) rectangle (6,3) ;
   \fill[green!50!white] (4,3) rectangle (6,4) ;
   \fill[green!50!white] (5,4) rectangle (6,5) ;
   \fill[green!50!white] (3,4) rectangle (4,5) ;
   \fill[green!50!white] (1,6) rectangle (2,7) ;
   \draw[step=1cm,gray!20!white] (0,-1) grid (8,7);
   \draw[step=2cm, black, shift={(0,-3)}] (0,2) grid (8,10);
}}
\caption{$A_{\mathbf{b_1b_1}} = (C_1)_{\mathbf{b_1b_1}} + F_1 $}
\label{fig:st0}
\end{subfigure}%
\begin{subfigure}[t]{.3\textwidth}
\centering
\resizebox{.9\textwidth}{!}{
 \tikz{
   \draw[step=1cm,white] (-1,-5) grid (6,2);
   \fill[blue!50!white] (0,0) rectangle (1,1);
   \node [right] at (1.5,0.5) {\Large Matrix $(\widehat{C}_1)_{\mathbf{b_1b_1}} $ };
   \fill[blue!20!white] (0,-1) rectangle (1,-2);
   \node [right] at (1.5,-1.5) {\Large  Matrix $F_1$};
   \fill[green!50!white] (0,-4) rectangle (1,-3);
   \node [right] at (1.5,-3.5) {\Large \text{ Matrix $\widehat{F}_{\mathbf{ml}1}$ }};
}}
 \label{fig:1}
\end{subfigure}%
\caption{Small ($r$-size) and big ($B$-size) far and close blocks of matrix $A_{\mathbf{b_1b_1}}$}
\label{fig:join}
\end{figure} $\\$
Similarly, for the matrix $\widetilde{A}_1$:
$$\widetilde{A}_1 = \widehat{C}_1 + \widehat{F}_1 = \widehat{C}_1 + \widehat{F}_{\mathbf{ml}1} + F_1 =
C_1 + F_1.$$
It can be shown that block rows and columns of the matrix $F_1$ have low-rank by the properties of the \HT~matrix $A$.
Similarly to \eqref{eq:q} compute orthogonal
block-diagonal matrices $\\U_{\mathbf{b_1}}, V_{\mathbf{b_1}} \in \R{Mr}{Mr}$ that compress matrix $F_1$.

Multiplication of matrix $F_1$  by matrices $U_{\mathbf{b_1}}$ and $ V_{\mathbf{b_1}}$ leads to compression:
\begin{equation}\widetilde{F}_2 = U_{\mathbf{b_1}}^{\top} F_1 V_{\mathbf{b_1}},
\label{eq:comp0}
\end{equation}
where the matrix $\widetilde{F}_2$ consists of compressed blocks.

Now we introduce extended matrices  $U_{\mathbf{b_1}}$ and $V_{\mathbf{b_1}}$ that can be applied to matrix $\widehat{A}_1$:
\begin{equation} U_1 =
\begin{bmatrix}
I_{(N-Mr) \times (N-Mr)}&0  \\
0&U_{\mathbf{b_1}} \\
\end{bmatrix}, \quad
V_1 =
\begin{bmatrix}
I_{(N-Mr) \times (N-Mr)}&0  \\
0&V_{\mathbf{b_1}} \\
\end{bmatrix}.
\label{eq:u1v1}
\end{equation}

Applying matrices $U_1$ and $V_1$ to matrix $\widetilde{A}_1$ we obtain the matrix with compressed first level:
$$A_2 = U_1^{\top} \widetilde{A}_1 V_1.$$
The process of the first level compression is shown in Figure \ref{fig:a1h}.
\begin{figure}[H]
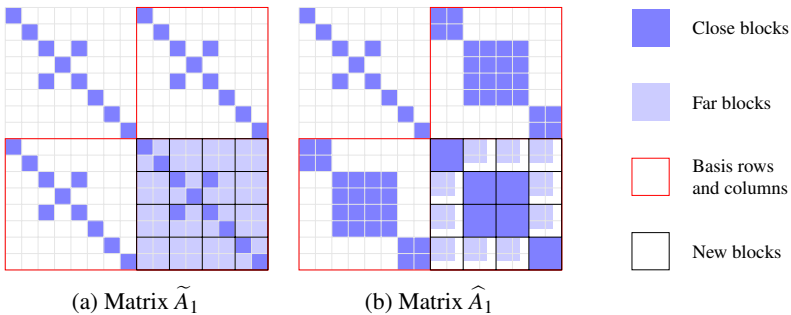

\centering
\begin{subfigure}[t]{.3\textwidth}
\centering
\resizebox{.9\textwidth}{!}{
\tikz{
   \foreach \i in {-8,...,-1}
   {
       \fill[blue!50!white] (\i,7-\i) rectangle (\i+1,8-\i) ;
   }
   \fill[blue!50!white] (4-8,13) rectangle (5-8,14) ;
   \fill[blue!50!white] (2-8,11) rectangle (3-8,12) ;

   \foreach \i in {-7,...,0}
   {
       \fill[blue!50!white] (\i-1,-\i) rectangle (\i,1-\i) ;
   }
   \fill[blue!50!white] (4-8,4+1) rectangle (5-8,5+1) ;
   \fill[blue!50!white] (2-8,2+1) rectangle (3-8,3+1) ;
   \foreach \i in {-7,...,0}
   {
       \fill[blue!50!white] (\i+7,8-\i) rectangle (\i+8,9-\i) ;
   }
   \fill[blue!50!white] (4,4+9) rectangle (5,5+9) ;
   \fill[blue!50!white] (2,2+9) rectangle (3,3+9) ;
   \fill[blue!20!white] (0,0) rectangle (8,8) ;
   \foreach \i in {0,...,7}
   {
       \fill[blue!50!white] (\i,7-\i) rectangle (\i+1,8-\i) ;
   }
   \fill[blue!50!white] (4,4+1) rectangle (5,5+1) ;
   \fill[blue!50!white] (2,2+1) rectangle (3,3+1) ;
   \draw[step=1cm,gray!20!white] (-8,0) grid (8,16);
   \draw[red] (-8,0) rectangle (8,8);
   \draw[red] (0,0) rectangle (8,16);
   \draw[step=2cm,black] (0,0) grid (8,8);
}}
\caption{Matrix $\widetilde{A}_1$}
\label{fig:a1}
\end{subfigure}%
\centering
\begin{subfigure}[t]{.3\textwidth}
\centering
\resizebox{.9\textwidth}{!}{
\tikz{
   \foreach \i in {-8,...,-1}
   {
       \fill[blue!50!white] (\i,7-\i) rectangle (\i+1,8-\i) ;
   }
   \fill[blue!50!white] (4-8,13) rectangle (5-8,14) ;
   \fill[blue!50!white] (2-8,11) rectangle (3-8,12) ;

   \foreach \i in {0,...,3}
   {
       \fill[blue!50!white] (2*\i-8,6-2*\i) rectangle (2*\i+2-8,8-2*\i) ;
   }
   \fill[blue!50!white] (4-8,3+1) rectangle (6-8,5+1) ;
   \fill[blue!50!white] (2-8,1+1) rectangle (4-8,3+1) ;

   \foreach \i in {0,...,3}
   {
       \fill[blue!50!white] (2*\i,14-2*\i) rectangle (2*\i+2,16-2*\i) ;
   }
   \fill[blue!50!white] (4,10+2) rectangle (6,12+2) ;
   \fill[blue!50!white] (2,8+2) rectangle (4,10+2) ;

   \fill[blue!20!white] (0,0) rectangle (8,8) ;
   \foreach \i in {0,...,3}
   {
       \fill[white] (2*\i+1.5,0) rectangle (2*\i+2,8) ;
       \fill[white] (0,2*\i) rectangle (8,2*\i+0.5) ;
   }

   \draw[step=1cm,gray!20!white] (-8,0) grid (8,16);
   \foreach \i in {0,...,3}
   {
       \fill[blue!50!white] (2*\i,6-2*\i) rectangle (2*\i+2,8-2*\i) ;
   }
   \fill[blue!50!white] (4,3+1) rectangle (6,5+1) ;
   \fill[blue!50!white] (2,1+1) rectangle (4,3+1) ;
   \draw[red] (-8,0) rectangle (8,8);
   \draw[red] (0,0) rectangle (8,16);
   \draw[step=2cm,black] (0,0) grid (8,8);
}}
\caption{Matrix $\widehat{A}_1$}
\label{fig:a1h}
\end{subfigure}%
\begin{subfigure}[t]{.3\textwidth}
\centering
\resizebox{.9\textwidth}{!}{
 \tikz{
   \draw[step=1cm,white] (-1,-6) grid (6,1);
   \fill[blue!50!white] (0,0) rectangle (1,1);
   \node [right] at (1.5,0.5) {\Large Close blocks };
   \fill[blue!20!white] (0,-1) rectangle (1,-2);
   \node [right] at (1.5,-1.5) {\Large Far blocks};
   \draw[red] (0,-4) rectangle (1,-3);
   \node [right] at (1.5,-3.2) {\Large Basis rows };
   \node [right] at (1.5,-3.8) {\Large and columns};
   \draw[black] (0,-6) rectangle (1,-5);
   \node [right] at (1.5,-5.5) {\Large New blocks};
}}
\end{subfigure}%
\caption{The first level compression}
\label{fig:c_alg0}
\end{figure}
For the first level we obtain
$$ A_2 = U_1^{\top} (C_1 + F_1) V_1 = U_1^{\top}C_1V_1 + \widehat{F}_2.$$
\subsection{Compression at all levels}
We apply permutation and repeat this procedure $L$ times and obtain:
\begin{equation}\label{eq:fin}
\begin{gathered}
 A_1 = U_0^{\top} A V_0 = U_0^{\top} C_0 V_0 + \widehat{F}_1 \\
 A_2 = U_1^{\top} U_0^{\top} A V_0 V_1  =  U_1^{\top} C_1 V_1 + \widehat{F}_2 \\
 \vdots \\
 A_L = \left(\prod_{k=0}^LU_k^{\top}\right)A \left(\prod_{k=L}^0V_k\right) = U_{L-1}^{\top} C_{L-1} V_{L-1} + \widehat{F}_L = S,\\
\end{gathered}
\end{equation}
thus
$$A = \left(\prod_{k=L}^0U_k\right)S \left(\prod_{k=0}^LV_k^{\top} \right).$$
If we denote
\begin{equation}U = \prod_{k = 0}^{L}U_{k},\quad V = \prod_{k = 0}^{L}V_k,
\label{eq:main_uv}
\end{equation}
then the final result of the algorithm is a sparse approximate factorization
\begin{equation}
A = U S V^{\top},
\label{eq:main_fact}
\end{equation}
where $S$ is a sparse matrix of the same size as matrix $A$, $U$ and $V$ are orthogonal
matrices that are products of permutation and block-diagonal orthogonal
matrices.

\begin{remark}
If matrix $A$ is approximated into \HT~format, then matrices $S$, $U$ and $V$ can be constructed from parameters of the \HT~representation, see details in Section~\ref{sec:h2tosp}.
\end{remark}

\begin{remark}
Sparsity of the matrix $S$ is proven in Section \ref{sec:sp}.
\end{remark}
\begin{proposition}
If the matrix $A$ is symmetric and positive definite, then the factors $U$ and $V$ are equal and the matrix $S$ is symmetric and positive definite.
\end{proposition}
\begin{proof}
If the matrix $A$ is symmetric and positive definite, then from the steps of compression algorithm, compression matrices $U_i$ and $V_i$, $i\in 0,\dots,L$ are equal, thus, by equations~\eqref{eq:main_uv}, $U = V$. Since $S = U^{\top}AV$, $U=V$ and $A$ is symmetric and positive definite, then $S$ is symmetric and positive definite matrix.
\end{proof}
\begin{remark}
The proposed sparsification algorithm is applicable to the special cases of \HT~matrices such as HSS {\color{royalblue}(Hierarchically Semiseparable)}, HOLDR {\color{royalblue}(Hierarchical Off-Diagonal Low-Rank)}.
\end{remark}

\subsection{Pseudo code of the compression algorithm}$\\$
\begin{algorithm}[H]
\label{al:comp}
\SetKwProg{D}{Input:}{}{}
\SetKwProg{I}{Initialization:}{}{}
\SetKwProg{ol}{Compression:}{}{}
\SetKwFunction{Out}{Output:}
\SetKwProg{fact}{Factorization:}{}{}
\SetKwProg{bs}{Backward step:}{}{}
\SetKwProg{O}{Output:}{}{}
\SetKw{ret}{return:}
\caption{Compression algorithm}
\D{}{$A \in \mathbb{R}^{N \times N}$ - matrix with \HT~structure\\
$L$ - number of levels\\
}
\ol{}{
   \For{$k = 0, \dots, L$}{
       $M_k$ - number of blocks on level $M_k$\\
       $P_{rk}, \, P_{ck}$ basis-non-basis permutations\\
       $\widetilde{A}_k = P_{rk}A_1P_{ck}$, ($P_{r0} = I, P_{c0}=I$)\\
       \For{$i = 1,\dots,M_k$}{
           Compute $U_{ki}$ (using SVD of $i$-th  block row of $A_{\mathbf{b_kb_k}}$)\\
           Compute $V_{ki}$ (using SVD of $i$-th  block column of $A_{\mathbf{b_kb_k}}$)\\
       }
       $U_k = P_{rk}\,\mathbf{diag}(U_{k1}, \dots, U_{kM_k})$\\
       $V_k = P_{ck}\,\mathbf{diag}(V_{k1}, \dots, V_{kM_k})$\\
       $A_{j+1} = U_jA_jV_j^{\top}$\\
   }

}

\O{Factorization  $A \approx USV^{\top}$}{
$U = \left(\prod_{k=0}^{L} U_k\right),$\\
$V = \left(\prod_{k=0}^{L} V_k\right),$\\
$S = A_L = \left(\prod_{k=0}^LU_k^{\top}\right)A \left(\prod_{k=L}^0V_k\right)$ \\
}
\end{algorithm}

\section{Sparsity of the matrix $S$}
\label{sec:sp}

First, define the block sparsity pattern of a block sparse matrix. For a matrix~$A$ with $M_1$ block columns, $M_2$ block rows and block size $B$ define $\bsp{A}{M_1\times M_2}{B \times B}$ (block sparsity pattern) as a function
$$\textbf{bsp}: \mathbb{R}^{M_1B\times M_2B} \rightarrow \mathbb{B}^{M_1\times M_2},$$
where $\mathbb{B} = \left \{ 0,1\right\}$. The function takes block matrix $A\in \mathbb{R}^{M_1B\times M_2B}$ as input and returns as output the matrix $R = \bsp{A}{M_1\times M_2}{B \times B} \in \mathbb{B}^{M_1\times M_2}$ such that
$$ \left\{\begin{matrix}
R_{ij} = 1, \quad \text{if} \quad A_{ij}\in \mathbb{R}^{B\times B}\quad \text{is~nonzero~block}, \\
R_{ij} = 0, \quad \text{if} \quad A_{ij}\in \mathbb{R}^{B\times B}\quad \text{is~zero~block.~~~~~}
\end{matrix}\right. $$
By $\#\bsp{A}{M_1\times M_2}{B \times B}$ define the number of nonzero blocks of matrix~$A$ and the number of ones in matrix $R$.
\begin{proposition}
If the matrix $A$ has \HT~structure, the compression Algorithm~\ref{al:comp} has $L$ levels, the block size on each level is $B$, the matrix $A$ has zero level close matrix $C$, and the far blocks are compressed with rank $r = B/2$, then the
compression algorithm for the matrix $A$ leads to the factorization:
$$A  = U^{\top}SV$$ where $U$ and $V$ are orthogonal
matrices equal to the multiplication of block-diagonal compression and permutation matrices
$$U= \left(\prod_{j=0}^{K} U_jP_j^{\top}\right), \quad V = \left(\prod_{j=0}^{K} V_jP_j^{\top}\right),$$
$S$ is a \textbf{sparse} matrix that has
$$ \#S \leqslant \left(4L + 6(\frac{1}{2^L}-1)\right) \#\bsp{C}{M \times M}{B\times B}$$
nonzero blocks\footnote{The symbol \# before the matrix means the number of nonzero $(r \times r)$ blocks in this matrix.} of size $(r \times r)$.
\begin{proof}
Consider the matrix $S$ from~\eqref{eq:main_fact}. Let matrix $S_{ij}$ correspond to $i$-th level non-basis hyper row and $j$-th level non-basis hyper column.
Thanks to basis-non-basis row and column permutations $P_{ri}$ and $P_{ci}$ matrix $S$ is separated into blocks $S_{ij}$, where $i,j \in {0,\dots, L}$.

The number of nonzero blocks in $S$ is equal to sum of nonzero blocks in $S_{ij}$:
$$\#S  = \sum_{i = 0}^{L} \sum_{j=0}^{L} \#S_{ij}.$$
Let us compute the number of nonzero blocks in block $S_{ij}$. Since on each level we join block rows by groups of $J$ blocks, we obtain:
$$ \bsp{S_{ij}}{M/2^{i} \times M/2^{j}}{r\times r} =  \bsp{C}{M/2^{i} \times M/2^{j}}{2^{i}B\times 2^{j}B},$$
where $i,j \in 0, \dots, L$.

Note that
$$\bsp{C}{M/2^{i} \times M/2^{j}}{2^{i}B\times 2^{j}B} \leqslant \frac{\#\bsp{C}{M \times M}{B\times B}}{2^{\min(i,j)}}$$

Thus
$$\#S  \leqslant \sum_{i=0}^{L} \sum_{j=0}^{L} \#\bsp{C}{M/2^{i} \times M/2^{j}}{2^{i}B\times 2^{j}B} =\sum_{i=0}^{L} \left(\frac{2(L-i)-1}{2^i}\right)\#\bsp{C}{M \times M}{B\times B}=$$
$$ = \left(4L + 6(\frac{1}{2^L}-1)\right)\#\bsp{C}{M \times M}{B\times B}.$$
Obtain
$$ \#S \leqslant \left(4L + 6(\frac{1}{2^L}-1)\right) \#\bsp{C}{M \times M}{B\times B}.$$
Thus, the number of nonzero blocks in matrix $S$ is less than the number of nonzero blocks in close matrix $C$ multiplied by constant ($4L + 6(\frac{1}{2^L}-1)$), if matrix $C$ is block-sparse, and all proposition conditions are met, then matrix $S$ is also sparse.
\end{proof}

\end{proposition}

\section{Building sparse factorization from \HT~coefficients}
\label{sec:h2tosp}

\subsection{Definition of \HT~matrix}
\label{sec:ht}
In this section we consider the matrix $A$ approximated in \HT~format. There exists a number of efficient ways to build
this approximation\cite{mikhalev-h2tools-2016,Borm-h2-2010}. In this paper, we do not consider the process of building the \HT~matrix and assume that it is given.
Let us explain in details how to construct factors in the decomposition~\eqref{eq:main_fact} from parameters of \HT~matrix $A$.
First, we give the definition of \HT~matrix, the more detailed definition can be found in\cite{Borm-h2-2010}.
{\color{royalblue}

\begin{definition}[Row and column cluster trees]{}
\label{def:ct}
Cluster trees of rows and columns $\mathcal{T}_r$ and $\mathcal{T}_c$ define the hierarchical division of block rows and columns.
At each level of the row cluster tree $\mathcal{T}_r$, each node corresponds to a block row of the matrix $A$, child nodes correspond to the subrows of this row. Same for the column cluster tree $\mathcal{T}_c$. 
\end{definition}

\begin{definition}[Block cluster tree]{}
Let $\mathcal{T}_{rc}$ be a tree. $\mathcal{T}_{rc}$ is a block cluster tree for $\mathcal{T}_r$ and $\mathcal{T}_c$ if it satisfies the following conditions:
\begin{itemize}
   \item $\text{root}(\mathcal{T}_{rc}$) = ($\text{root}(\mathcal{T}_{r}$),$\text{root}(\mathcal{T}_{c}$)).
   \item Each node $b\in\mathcal{T}_{rc}$ has the form $b=(t,s)$ for $t\in\mathcal{T}_r$ and $s\in\mathcal{T}_c$.
   \item Let $b=(t,s)\in\mathcal{T}_{rc}$. If $\text{sons}(b)\neq \varnothing$, then
   $$\text{sons}(b) = \begin{cases} \{t\}\times\text{sons}(s) \quad \text{~if~sons}(t) \neq \varnothing,\text{sons}(s) \neq \varnothing,\\
                                    \text{sons}(t)\times\{s\} \quad \text{~if~sons}(t) \neq \varnothing,\text{sons}(s) \neq \varnothing,\\
                                    \text{sons}(t)\times\text{sons}(s) \quad otherwise.
   \end{cases}$$
\end{itemize}

\end{definition}

\begin{definition}[Admissibility condition]{}
\label{def:com}
Let $\mathcal{T}_r$ and $\mathcal{T}_r$  be a row and column cluster trees. A predicate
$$\mathcal{A} = \mathcal{T}_r \times \mathcal{T}_c \xrightarrow{} \{\text{True,False}\}$$
is an admissibility condition for $\mathcal{T}_r$ and $\mathcal{T}_r$ if
$$ \mathcal{A}(t,s) \Longrightarrow \mathcal{A}(t',s) \quad \text{~holds~for~all~} \quad t\in \mathcal{T}_r, s\in\mathcal{T}_c,t'\in\text{sons}(t)$$
and
$$ \mathcal{A}(t,s) \Longrightarrow \mathcal{A}(t,s') \quad \text{~holds~for~all~} \quad t\in \mathcal{T}_r, s\in\mathcal{T}_c,s'\in\text{sons}(s).$$
If $\mathcal{A}(t,s)$, the pair (t,s) is called admissible.
\end{definition}

\begin{definition}[Admissibility block cluster tree]{}
   Let $\mathcal{T}_{rc}$ be a block cluster tree for $\mathcal{T}_r$ and $\mathcal{T}_c$, let $\mathcal{A}$ be an admissibility condition. If for each
   $(t,s)\in\mathcal{T}_{rc}$ either $\text{sons}(t)\neq \varnothing \neq\text{sons}(s)$ or $ \mathcal{A}(t,s) = $ True holds, the block cluster tree $\mathcal{T}_{rc}$ called $ \mathcal{A}$-Admissible.
\end{definition}

\begin{definition}[Farfield and nearfield]{}
   Let $\mathcal{T}_{rc}$ be a block cluster tree for $\mathcal{T}_r$ and $\mathcal{T}_c$, let $\mathcal{A}$ be an admissibility condition. The index set
   $$ \mathcal{I}_{N\times N}^{+} := \{(t,s) \in\mathcal{I}_{N\times N} \quad : \mathcal{A}(t,s) = \text{True}\}$$
   is called set of farfield blocks. The index set
   $$ \mathcal{I}_{N\times N}^{-} := \{(t,s) \in\mathcal{I}_{N\times N} \quad : \mathcal{A}(t,s) = \text{False}\}$$
   is called set of nearfield blocks.
\end{definition}

\begin{definition}[Cut-off matrices]{}
\label{def:com}
Let $\mathcal{T}_r$ be a row cluster tree. For all tree nodes $t\in \mathcal{T}_r$ the cut-off matrix $\chi_t\in\R{N}{N}$ corresponding to $t$ is defined by
$$ (\chi_t)_{ij} = \begin{cases} 1 \quad \text{~if~} i=j \in t,\\ 0 \quad\text{~otherwise},\end{cases} \text{~for~all~} i,j \in N.$$
\end{definition}

\begin{definition}[Cluster basis]{}
\label{def:cb}
Let $K = (K_t)_{t\in\mathcal{T}_r}$ be a family of finite index sets ({\it rank distribution for} $\mathcal{T}_r$). Let $R = (R_t)_{t\in\mathcal{T}_r}$ be a family of matrices satisfying $R_t\in\mathbb{R}_{\hat{t}}^{N\times K_t}$ for all $t\in \mathcal{T}_r$. Then $R$ is called row cluster basis and the matrices $R_t$ are called row cluster basis matrices. Analogically for column cluster basis $E$ and column cluster
basis matrices $E_s$,  $s\in \mathcal{T}_c$.
\end{definition}

\begin{definition}[Close matrix]{}
\label{def:close}
The matrix $C \in \mathbb{R}^{N \times N}$ is close if
$$C = \sum_{b = (s,t)\in\mathcal{I}_{N\times N}^{-}} \chi_t A \chi_s$$
\end{definition}

\begin{definition}[$\mathcal{H}^2$ matrix]{}
\label{def:ht}
The matrix $A \in \mathbb{R}^{N \times N}$ is approximated in \HT~format if $\mathcal{T}_c$ and $\mathcal{T}_r$ are block cluster trees of columns and rows of matrix $A$, if there exist a row cluster basis $R$ (row transition matrices), the column cluster basis $E$ (column transition matrices), a family $D = (D_b)_{b\in\mathcal{I}_{N\times N}^{+}}$ of matrices satisfying $D_b \in \R{K_t}{L_s}$
for all $b=(s,t)\in \mathcal{I}_{N\times N}^{+}$
(interaction list), and close matrix $C$, and if
$$A = C + \sum_{b = (s,t)\in\mathcal{I}_{N\times N}^{+}} R_t \,D_b \,E^{*}_s.$$
\end{definition}
}

\subsection{Construction of matrices $U$ and $V$ from the coefficients of \HT~matrix}
\label{sec:su}
Let us first construct orthogonal
matrices $U$ and $V$ from the factorization~\eqref{eq:main_fact}. According to equation~\eqref{eq:main_uv}:
$$ U = \prod_{k = 0}^{L}P_{rk}U_{k},$$
and
$$V = \prod_{k = 0}^{L}P_{ck}V_k.$$
Note that matrices $U_k$ and $V_k, k \in 0,\dots, L$,  are very close in their meaning to cluster
basis matrices $R_t$ and $E_s$ both are level compression matrices. The difference between these matrices is that the diagonal blocks of matrices $U_k$ and $V_k$ are square orthogonal blocks, and diagonal blocks of matrices $R_t$ and $E_s$ are rectangular non-orthogonal
blocks.

Thus we can take matrices $R_t$ and $E_s$, orthogonalize blocks, complete each block to square orthogonal
block and obtain the matrices $U_k$ and $V_k$. The algorithm that orthogonalizes blocks of matrices $R_t$ and $E_s$ is
known as the \HT~compression algorithm, it can be found in\cite{Borm-h2-2010,borm-h2svd-2012}. Compression of blocks can be done by QR decomposition of blocks with the square $Q$ factor.
Permutations $P_{rk}$ and $P_{ck}$ can be constructed from the cluster trees.

\subsection{Construction of the matrix $S$ from the coefficients of \HT~matrix}
\label{sec:cs}

According to equations~\eqref{eq:fin}, equation~\eqref{eq:cl0} and equation~\eqref{eq:cl1}:
$$S = U_{L-1}^{\top}C_{L-1}V_{L-1} + \widehat{F}_L = $$
$$ = U_{L-1}^{\top}(\dots(U_1^{\top}(U_0^{\top} C_0 V_0 + \widehat{F}_{\mathrm{ml}1})  V_1 + \widehat{F}_{\mathrm{ml}2}) \dots)V_{L-1}+  \widehat{F}_L.$$

Construction of matrices $U_i$ and $V_i$ is shown in the previous subsection, matrix $C_0$ is stored in the \HT~matrix explicitly as close matrix $C$, matrices $\widehat{F}_{\mathrm{ml}i}$ are exactly matrices $D_i$ from interaction list, matrix $\widehat{F}_L$ is matrix $D_L$. Thus, matrix $S$ can be easily computed from the coefficients of the \HT~matrix.

\section{Numerical experiments}
\label{sec:num}
Sparsification algorithm is implemented in the Python programming language. For the \HT~matrix implementation we use the h2tools\cite{mikhalev-h2tools-2016} library. 
All computations are performed on MacBook Air with a 1.3GHz Intel Core i5 processor and 4
GB 1600 MHz DDR3 RAM.

{\color{royalblue} First, we numerically show that the matrix $S$ in factorization
\eqref{eq:main_fact} is indeed sparse. Then we give the timing and storage requirements of the sparse factorization. }
Thereafter we consider the sparse factorization of the \HT~matrix combined with a sparse direct solver as a direct solver for {\color{royalblue}the system with} the \HT~matrix and compare this approach with HODLR direct solver {\color{royalblue} and \HT-LU solver from \Hm2Lib  library}.
Finally, we consider sparse factorization with the matrix $S$ factorized by the ILUt method as a preconditioner for GMRES solver.

{\color{royalblue}We want to note that the sparsification algorithm \textbf{does not worsen} the accuracy of the \HT~approximation. It follows from the algorithm and it is confirmed in the experiments. Therefore, in the experiments below, we omit the accuracy of the sparse factorization and show only the \HT~approximation accuracy to avoid redundancy.}
{\color{royalblue}
\subsection{Sparsity of the factor $S$}
}

In Section \ref{sec:sp} we have studied the sparsity of the factor $S$ from the factorization \eqref{eq:main_fact} analytically. Here we present the numerical illustration that matrix $S$ is indeed sparse.
Tests are performed for two \HT~matrices.
\begin{example}
   {\color{royalblue}\HT~approximation of the matrix}
\label{ex:holdor1}
\begin{equation} A_{ij} =\begin{cases}
\frac{1}{|r_i - r_j|} & \text{if}~i \neq j\\
0, & \text{if}~i = j
\end{cases} ,
\label{eq:core1}
\end{equation}
where $r_i \in \mathbb{R}^2$ or $r_i \in \mathbb{R}^3$  is the position of the $i$-th element. Elements are randomly distributed in identity square $\Omega = [0,1]\times[0,1]$ in $\mathbb{R}^2$ case and in identity cube $\Omega = [0,1]\times[0,1]\times[0,1]$ in $\mathbb{R}^3$ case.
{\color{royalblue}The \HT~approximation accuracy is $\epsilon = 10^{-6}$.}
In Figure~\ref{fig:nnz} {\color{royalblue} the factor $S$ of the sparse factorization of the \HT~matrix with the core \eqref{eq:core1} is marked by ``inv''.}
\end{example}

\begin{example}
\label{ex:holdor2}
{\color{royalblue}\HT~approximation of the matrix}
\begin{equation} A_{ij} = 2\delta_{ij} + \exp(-||r_i - r_j||^2),
\label{eq:core2}
\end{equation}
where $r_i \in \mathbb{R}^2$ or $r_i \in \mathbb{R}^3$ is the position of the $i$-th element. Elements are randomly distributed in identity square $\Omega = [0,1]\times[0,1]$ in $\mathbb{R}^2$ case and in identity cube $\Omega = [0,1]\times[0,1]\times[0,1]$ in $\mathbb{R}^3$ case. {\color{royalblue} The \HT~approximation accuracy is $\epsilon = 10^{-6}$,}
In Figure~\ref{fig:nnz} {\color{royalblue}the factor $S$ of the sparse factorization of the \HT~matrix with the core \eqref{eq:core2} is marked} by ``exp''.
\end{example}

{\color{royalblue}
We consider the sparse factorization \eqref{eq:main_fact} of the \HT~matrices from Examples~\ref{ex:holdor1}~and~\ref{ex:holdor2} in 2D and 3D.
In Figure~\ref{fig:nnz} we show the number of nonzero elements \textbf{per row} in the factor $S$. Since for all considered matrices the number of nonzero elements per row is constant (and does not grow with the matrix size), we conclude that the matrix $S$ in factorization
\eqref{eq:main_fact} is indeed sparse.}
\begin{figure}[H]
 \centering
 \includegraphics[scale=0.34]{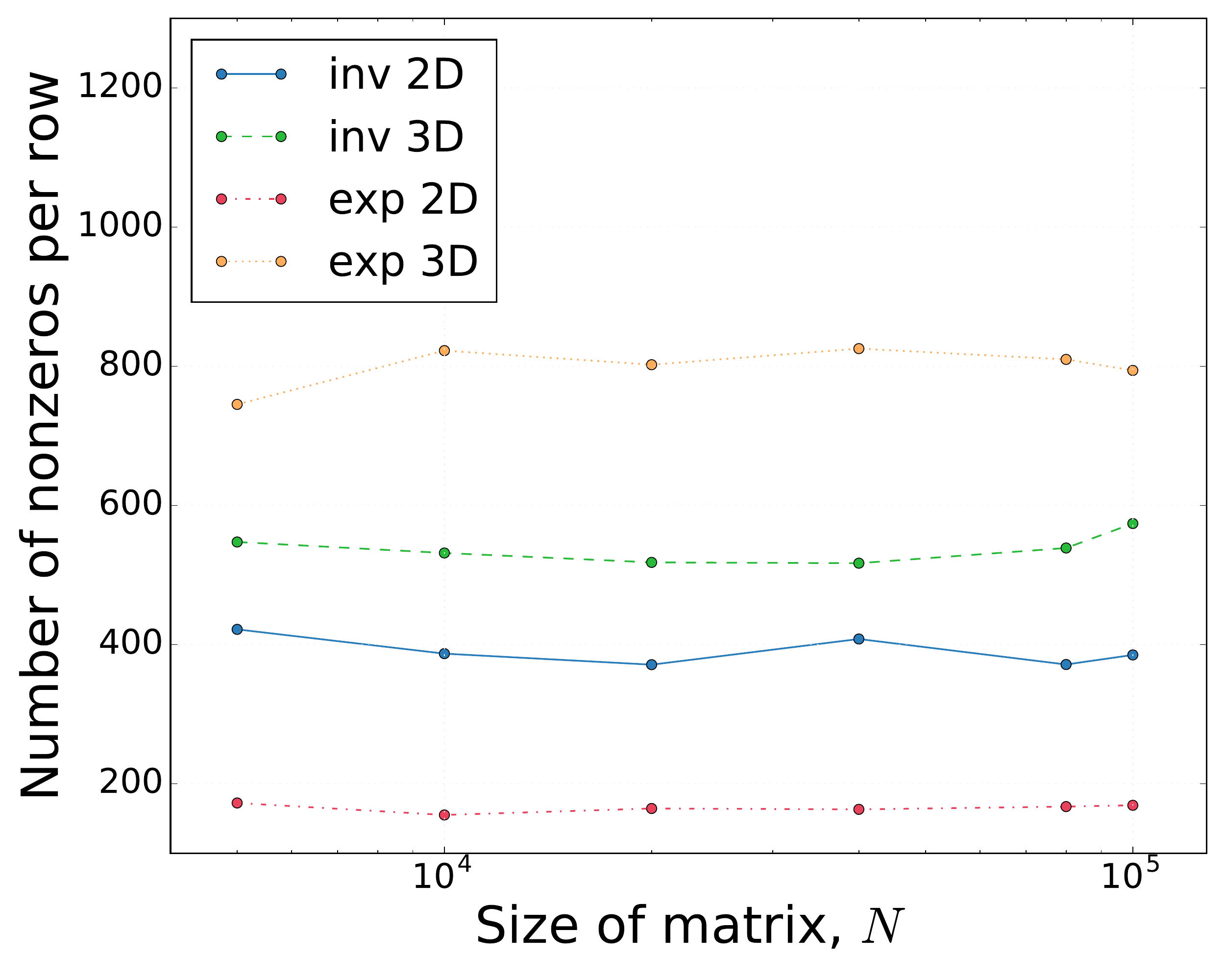}
 \caption{Number of nonzero elements per row {\color{royalblue}in factor $S$} for \HT~matrices from Examples~\ref{ex:holdor1}~and~\ref{ex:holdor2} in 2D and 3D.}
 \label{fig:nnz}
\end{figure}%

{\color{royalblue}
\subsection{The storage requirements and timing of the sparsification algorithm}
}

In Figure~\ref{fig:timing} we show the time of the sparse factorization of \HT~matrices with the cores \eqref{eq:core1} and \eqref{eq:core2} {\color{royalblue}(approximation accuracy is $\varepsilon = 10^{-6}$)} denoted by "inv" and "exp".
\begin{figure}[H]
 \centering
 \includegraphics[scale=0.34]{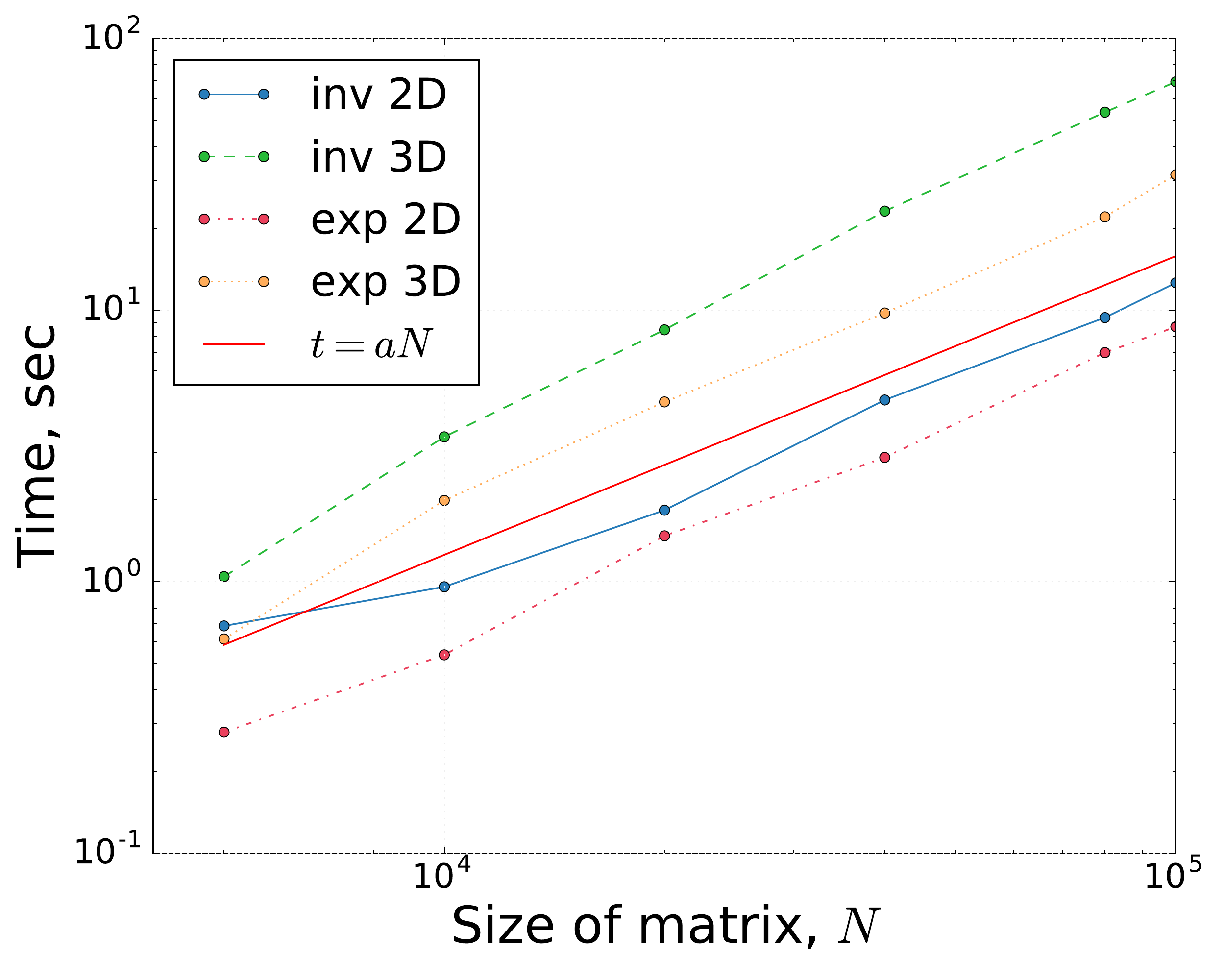}
 \caption{Timing of sparsification building for matrices from Examples~\ref{ex:holdor1}~and~\ref{ex:holdor2} in 2D and 3D.}
 \label{fig:timing}
\end{figure}%
Sparsification time grows almost linearly. {\color{royalblue}In Figure~\ref{fig:mem1} we show the memory requirements of the matrix \eqref{eq:core1} in 2D and 3D, also we show the memory requirements of its \HT~approximation ($\varepsilon = 10^{-6}$), of the sparse factorization (sum of $U,$ $S,$ and $V$) and of the factor $S$ separately.
\begin{figure}[H]
 \begin{subfigure}[t]{.5\textwidth}
   \centering
 \includegraphics[scale=0.24]{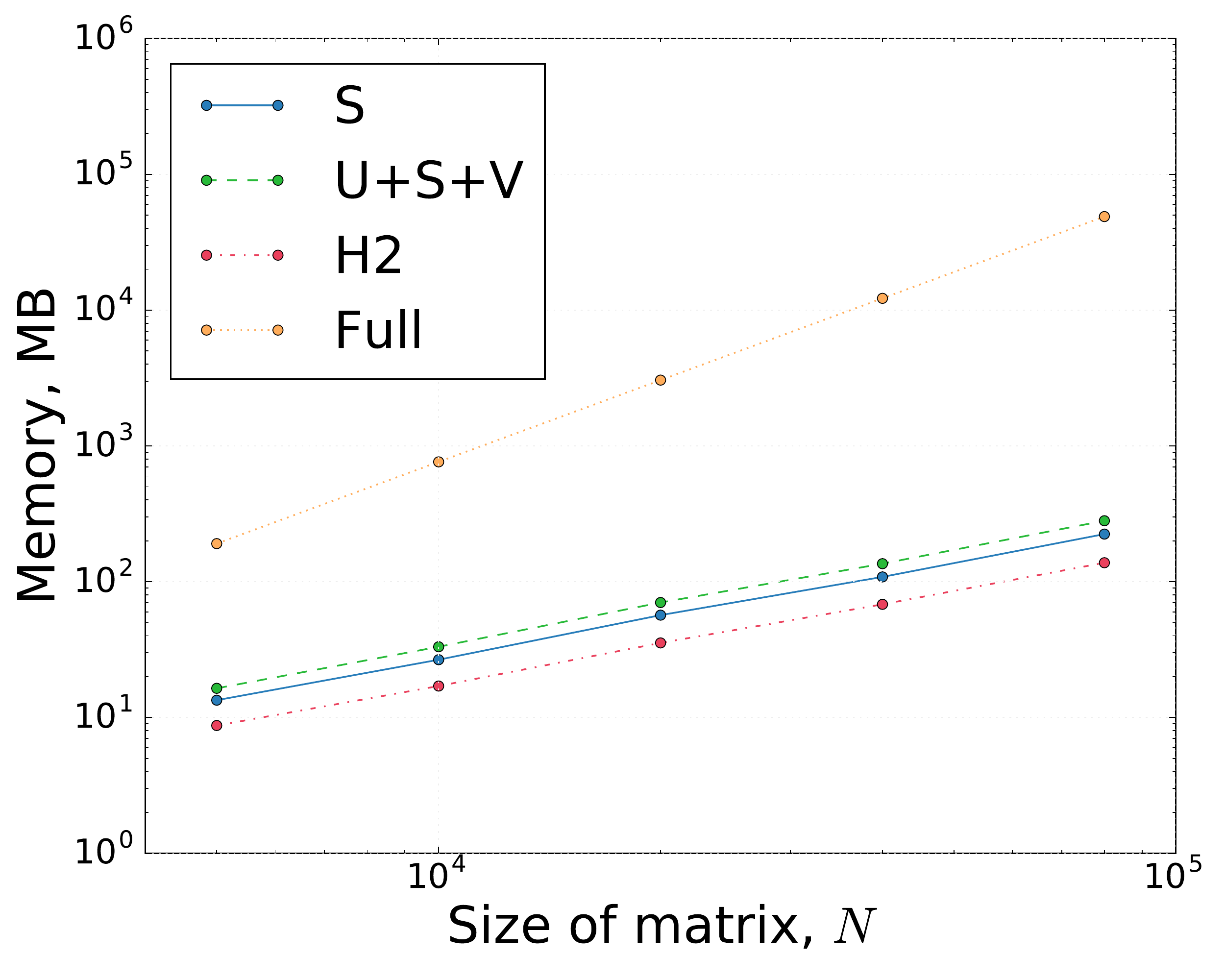}
 \caption{2D case}
 \end{subfigure}%
 \begin{subfigure}[t]{.5\textwidth}
   \centering
   \includegraphics[scale=0.24]{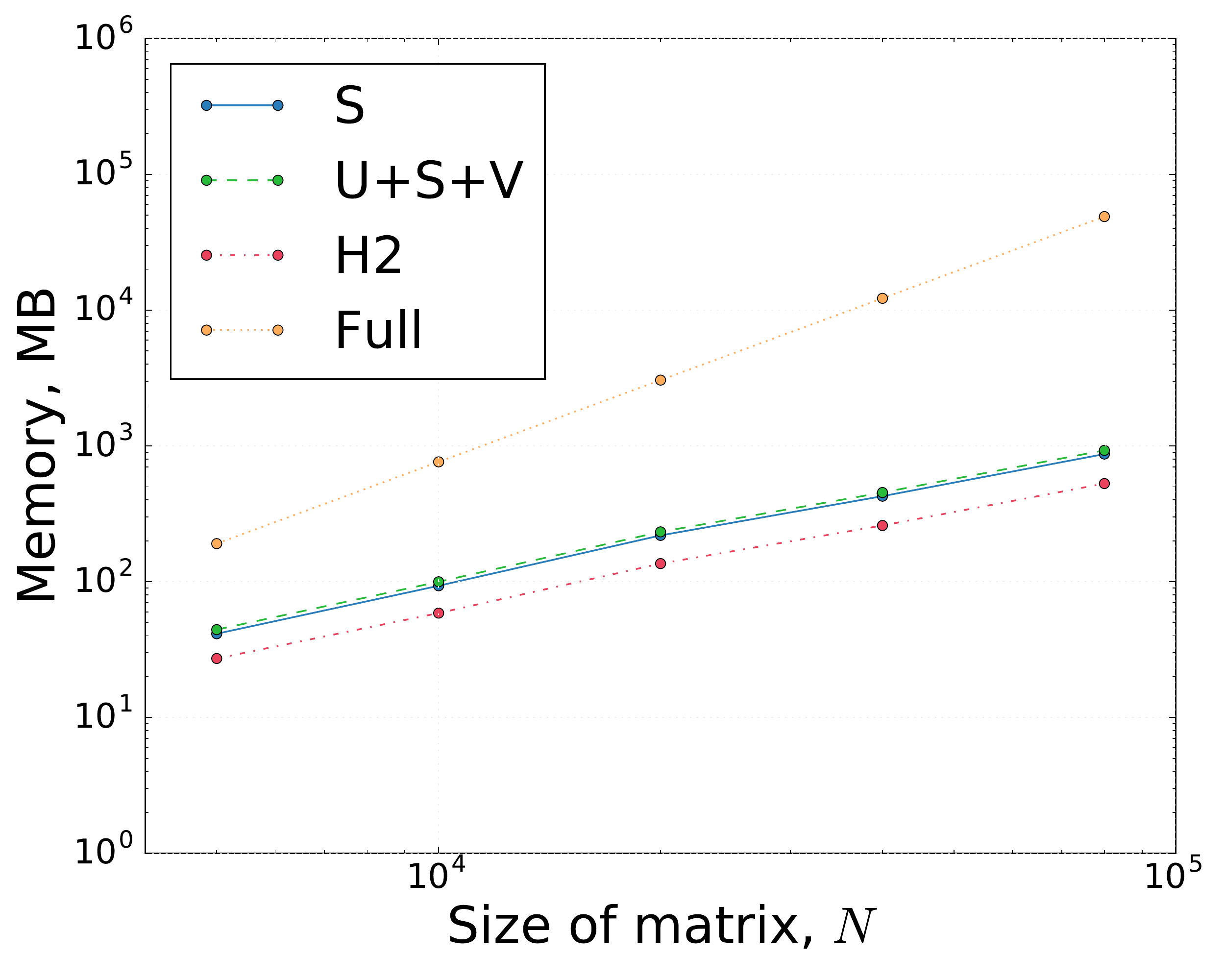}
   \caption{3D case}
 \end{subfigure}
 \caption{Storage requirements for the original matrix \eqref{eq:core1}, its \HT~approximation and the sparse factorization.}
 \label{fig:mem1}
\end{figure}%
In Figure~\ref{fig:mem2} we show the memory requirements of the matrix \eqref{eq:core2} in 2D and 3D, also we show the memory requirements of its \HT~approximation ($\varepsilon = 10^{-6}$), of the sparse factorization (sum of $U,$ $S,$ and $V$) and of the factor $S$ separately.
\begin{figure}[H]
 \begin{subfigure}[t]{.5\textwidth}
   \centering
 \includegraphics[scale=0.24]{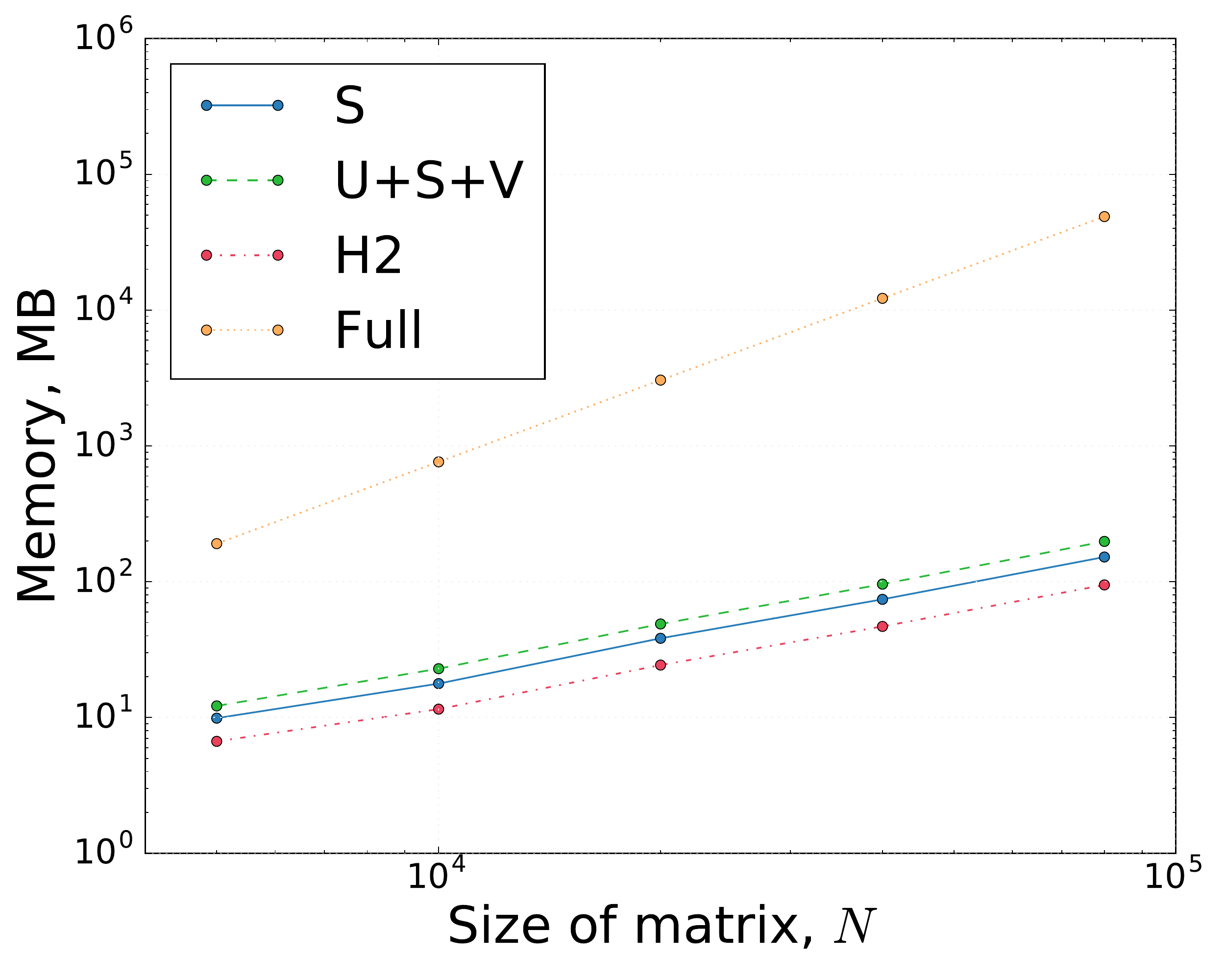}
 \caption{2D case}
 \end{subfigure}%
 \begin{subfigure}[t]{.5\textwidth}
   \centering
   \includegraphics[scale=0.24]{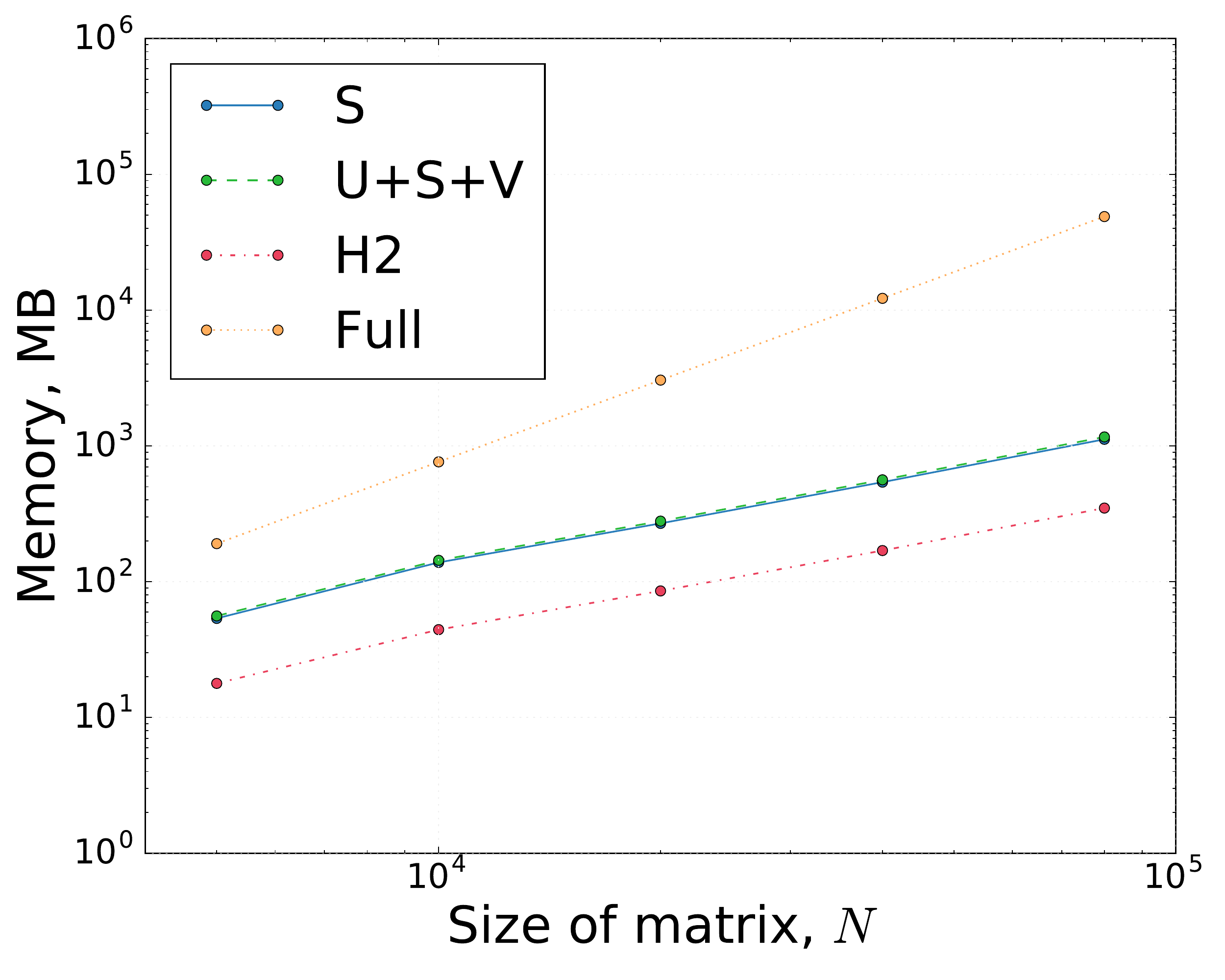}
   \caption{3D case}
 \end{subfigure}
 \caption{Storage requirements for the original matrix \eqref{eq:core2}, its \HT~approximation and the sparse factorization.}
 \label{fig:mem2}
\end{figure}%

For all examples, the memory requirements of both \HT~and the sparse factorization grows almost linearly, unlike the memory requirements of the original matrix which scales quadratically.}

\subsection{Comparison to HODLR}
Paper\cite{greengard-hodlr-2016} considers the HODLR approximation {\color{royalblue} of the dense matrix} and {\color{royalblue}its} factorization as an efficient way to compute the determinant of the dense matrix. We propose the \HT~approximation {\color{royalblue} of the dense matrix}, {\color{royalblue}its} sparsification and factorization of the sparse matrix as an alternative. {\color{royalblue}The triangular factorization of the sparse matrix is computed by CHOLMOD\cite{Davis-cholmod-2009} package.} Tests are performed for 3D data, for matrix
$$ A_{ij} = 2\delta_{ij} + \exp(-||r_i - r_j||^2),$$
where $r_i \in \mathbb{R}^3$ is the position of the $i$-th element. {\color{royalblue}Both HODLR and \HT~approximation accuracy is $\varepsilon = 10^{-6}$. The HODLR factorization accuracy is $\vartheta = 10^{-5}$, the accuracy of the triangular factorization of the sparse matrix is $\vartheta = 10^{-10}$ (which is redundant, but the used package has no accuracy options).}
In Figure~\ref{fig:hist_hodlr} we show the time comparison for this two approaches.
\begin{figure}[H]
 \centering
 \includegraphics[scale=0.34]{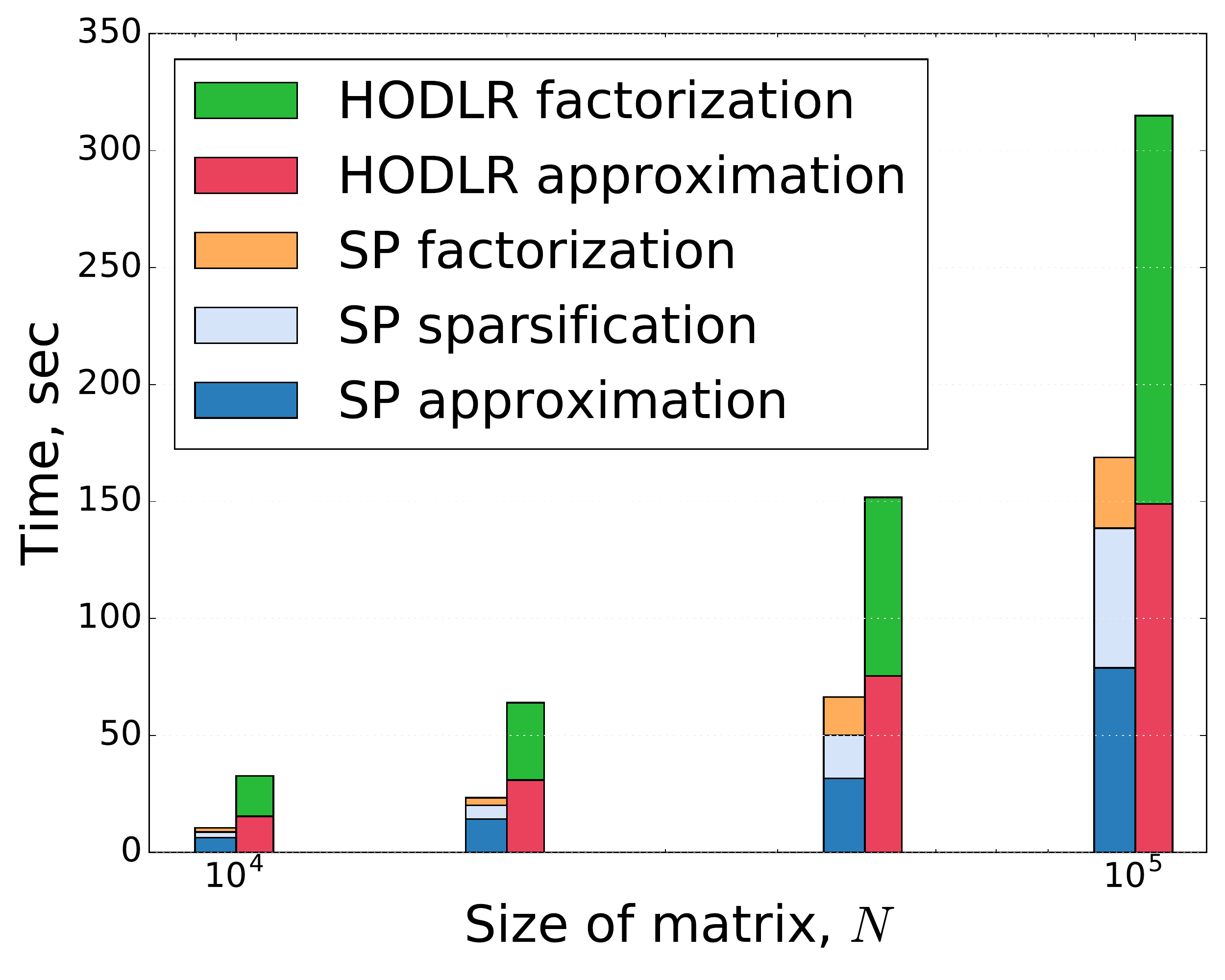}
 \caption{Comparison of the \HT~sparsification approach with HODLR solver in 3D.}
 \label{fig:hist_hodlr}
\end{figure}$\\$
Total solution time comparison in 2D and 3D is presented in Figure~\ref{fig:tot_hodlr}.
\begin{figure}[H]
 \centering
 \includegraphics[scale=0.34]{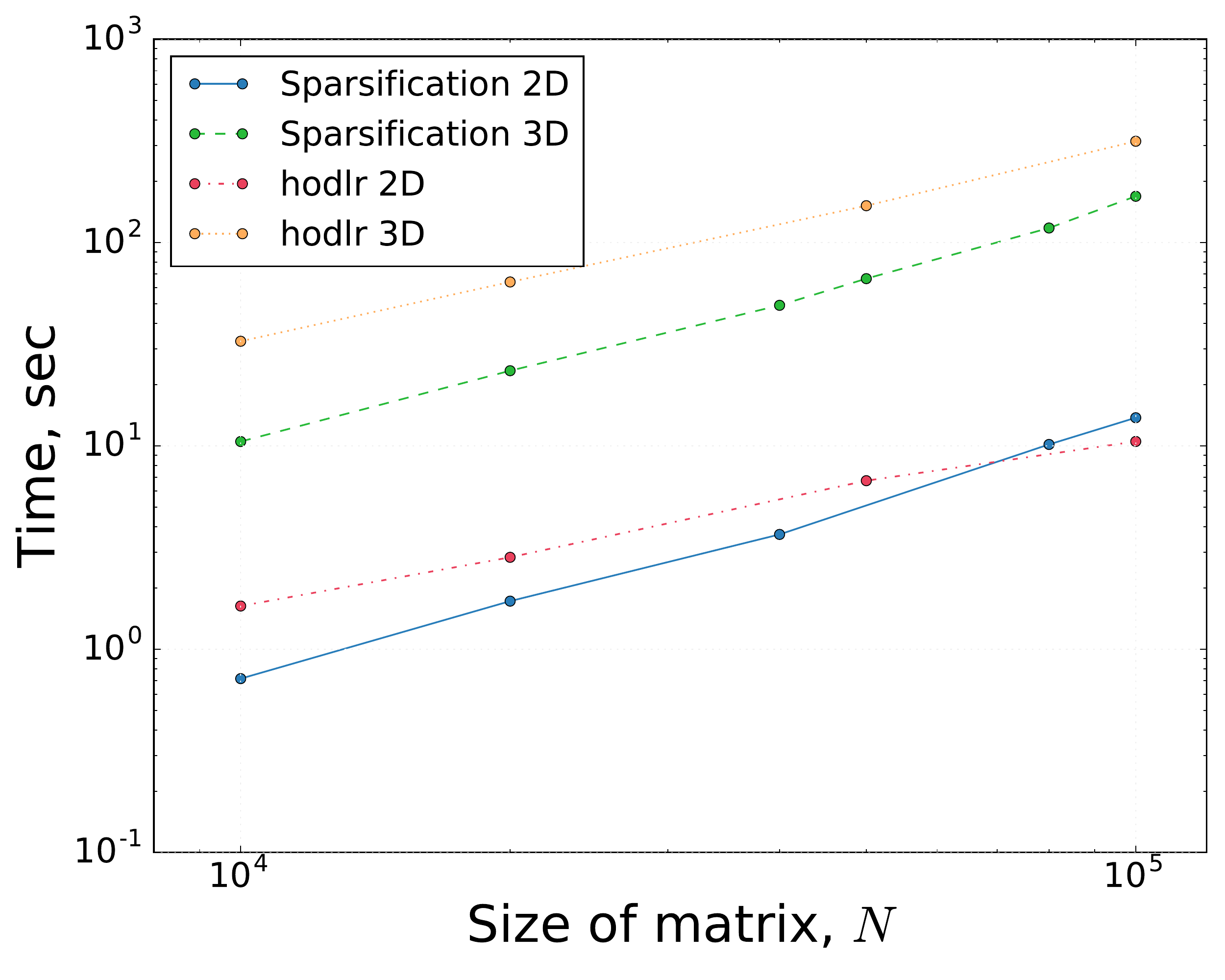}
 \caption{Comparison of the \HT~sparsification approach with HODLR solver in 2D and 3D.}
 \label{fig:tot_hodlr}
\end{figure}$\\$

\subsection{Comparison to H2Lib liberary}

In this subsection we compare the approach proposed in this paper (sparse non-extensive factorization of the \HT~matrix, triangular factorization of the sparse matrix and then the solution of the system) with the approach based on \HT-LU factorization, proposed in the work\cite{borm-h2lu-2013} and implemented in H2lib package\cite{Hakb-h2-lib}. The \HT-LU factorization takes the \HT~matrix and returns the triangular factors $L$ and $U$ in \HT~format. {\color{royalblue}The triangular factorization of the sparse matrix is computed by CHOLMOD\cite{Davis-cholmod-2009} package.} Tests are performed on the following problem.
\begin{example}
   \label{ex:h2_cube}
   Consider the Dirichlet boundary value problem for Laplace's equation
   \begin{equation}
       \label{eq:lap_bem_d}
       \begin{cases} -\Delta u(x) = 0  \quad x\in \Omega, \\ ~u(x) = f(x) \quad x\in  \Gamma =\partial\Omega, \end{cases}
   \end{equation}
   where $\Omega = [-1,1]^3$ is a cube.
   The standard technic: using the single layer potential we obtain boundary integral formulation of the equation~\eqref{eq:lap_bem_d}.
   \begin{equation}
       \label{eq:lap_bem_i}
       \int_{\Gamma } G(x-y)\varphi (y)\,ds_y = f(x),
   \end{equation}
   where $G(x) = \frac{1}{4\pi}\frac{1}{|x|} $ is the fundamental solution for the Laplace operator.
   Then we discretize the integral equation~\eqref{eq:lap_bem_i} on the triangular grid on $\Gamma$ using the Galerkin method.
   Obtained dense matrix is approximated in \HT~format with accuracy $\varepsilon = 10^{-6}$.
   \end{example}
{\color{royalblue}The accuracy of the \HT-LU factorization is $\vartheta = 10^{-6}$, the accuracy of the triangular factorization of the sparse matrix is $\vartheta = 10^{-10}$.}

In Table~\ref{tab:h2lib} we show the time comparison of the solution of the system with matrix from Example~\ref{ex:h2_cube}, using \HT-LU and sparsification approaches. For the \HT-LU we show the approximation in \HT~format and factorization time, for the sparsification we show the approximation in \HT~format, sparsification and sparse factorization time. In both cases, time of the solution of the system with factorized matrix is negligible, so we do not show it.
\begin{table}[H]
\centering
\begin{tabular}{ l | c  c  c  c }
 N & 3072&12288&49152&196608 \\
\hline
\hline
H2Lib approx.,~sec &6.72& 24.26&104.97&487.24\\
H2Lib factor.,~sec &13.56&164.00&1712.17&10970.43\\
\hline
Sp. approx.,~sec  & 4.8& 26.34&110.91&399.43 \\
Sp. sp.,~sec  & 2.16&10.12&63.47&323.23\\
Sp. factor.,~sec  & 0.19& 1.30& 7.83&56.89 \\
\end{tabular}
\caption{Comparison with H2Lib.}
\label{tab:h2lib}
\end{table}

In Figure~\ref{fig:comp_h2lib} we show the comparison of the total time, required for the system solution.

\begin{figure}[H]
 \centering
 \includegraphics[scale=0.34]{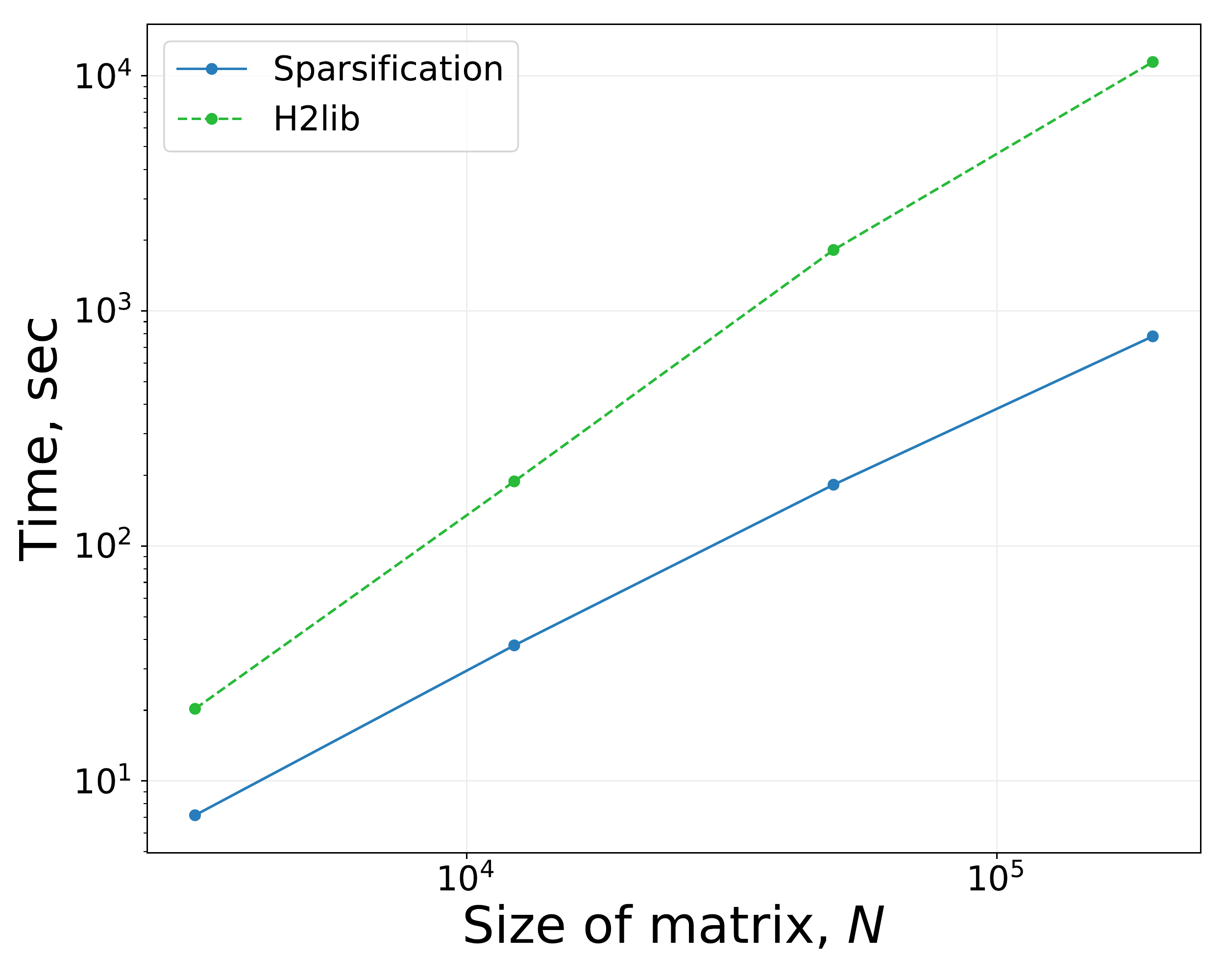}
 \caption{Total time required for the system solution using \HT~sparsification and H2lib solver.}
 \label{fig:comp_h2lib}
\end{figure}$\\$
The sparsification approach has not only better timing, but also better asymptotics.

\subsection{Sparsification method as a preconditioner}
\HT~matrix is an efficient tool to multiply a matrix by a vector. This allows to apply iterative solvers like GMRES to solution of the systems with \HT~matrix. But the preconditioning is still a challenging problem due to complexity of the factorization of the \HT~matrix. We propose to use the approximately factored sparsification of the \HT~matrix as a preconditioner to iterative method.
For tests we use randomly distributed 3D data with following interaction matrix.
\begin{example}
\label{ex:mat1}
$$ A_{ij} =\left\{\begin{matrix}
1 & \text{if}~i = j\\
\frac{|r_i - r_j|}{d} & \text{if}~0<|r_i - r_j|<d\\
\frac{d}{|r_i - r_j|}, & \text{if}~|r_i - r_j|\geqslant d
\end{matrix}\right. ,$$
where $r_i \in \mathbb{R}^3$ is the position of the $i$-th element.
\end{example}
This example is used for testing of IFMM {\color{royalblue}(Inverse Fast Multipole Method)} method as a preconditioning in\cite{darve-ifmm_prec-2015}, so we have chosen this example for the convenient comparison.

This matrix is useful for the iterative tests since condition number of this matrix significantly depends on the parameter $d$: the larger $d$ is, the larger condition number is.

\textbf{Example with well-conditioned matrix.}
First, consider the matrix from Example~\ref{ex:mat1} with $d = 10^{-3}$, condition number $\mathrm{cond}(A) = 10$. We solve this system using GMRES iterative solver for the matrix approximated in \HT~format with accuracy $\epsilon = 10^{-9}$ and as a preconditioner we use \HT~approximation of the matrix $A$ with accuracy $\epsilon = 10^{-3}$ sparsified and factorized with ILUt decomposition.  Figure~\ref{fig:it1_comv} shows convergence of the GMRES method with  different ILUt threshold parameters. The required residual of the GMRES method $ r = 10^{-10}$.
\begin{figure}[H]
 \centering
 \includegraphics[scale=0.34]{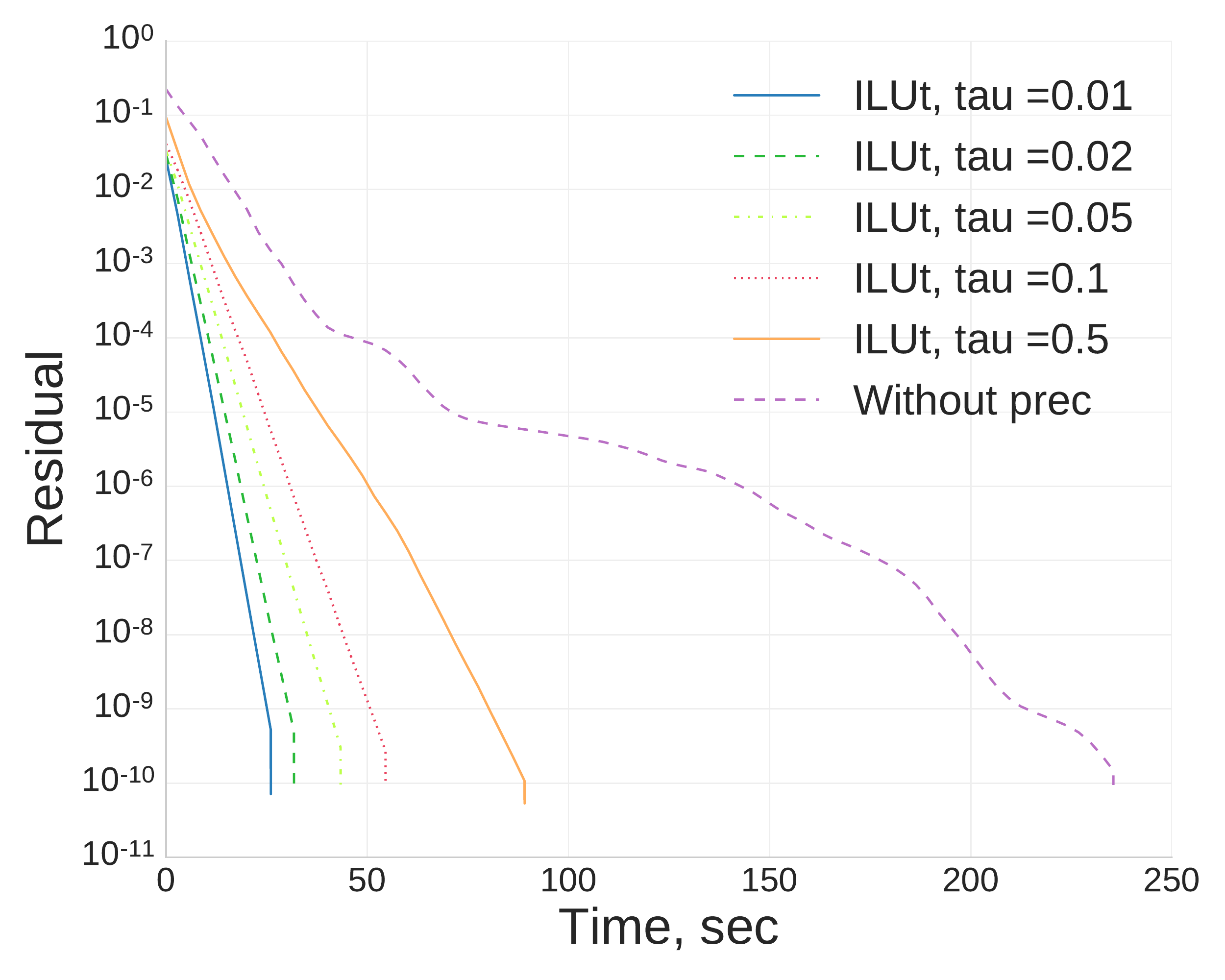}
 \caption{Convergence of GMRES with different drop tolerance parameter of the ILUt preconditioner}
 \label{fig:it1_comv}
\end{figure}%
The standard trade off: the more time on the preconditioner building we spend, the faster iterations converge.
Figure~\ref{fig:it1_time} illustrates the total time required for the system solution (including the sparsification construction).
\begin{figure}[H]
   \centering
     \includegraphics[scale=0.34]{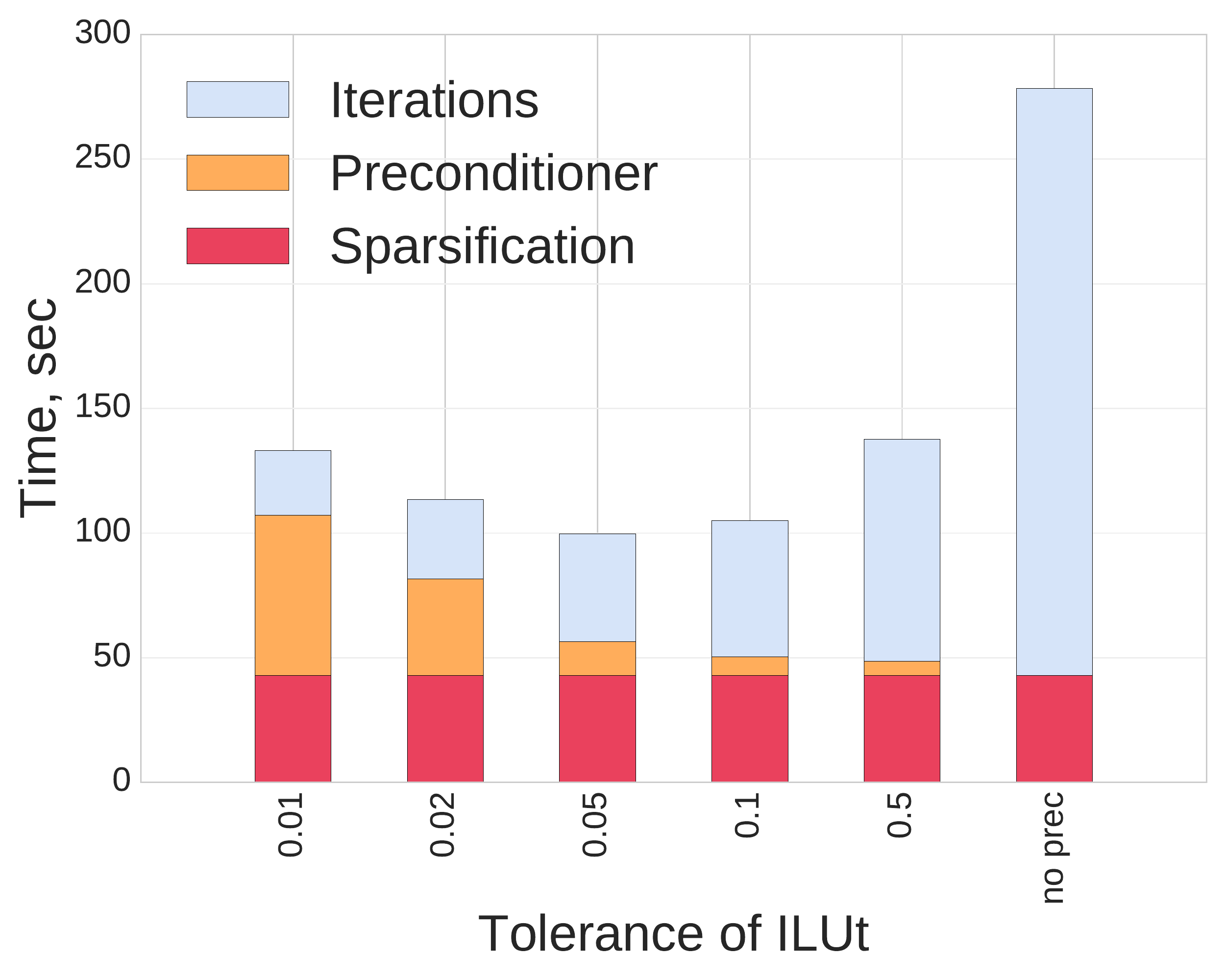}
     \caption{Contribution into total time of sparsification, factorization and iterations}
     \label{fig:it1_time}
\end{figure}%

We show the total time required for \HT~matrix sparsification, time required for building the ILUt preconditioner with dropping tolerance $\tau = 2 \times 10^{-2}$ and iterations timing in Figure~\ref{fig:it1_tot}.

\begin{figure}[H]
 \begin{subfigure}[t]{.5\textwidth}
   \centering
 \includegraphics[scale=0.24]{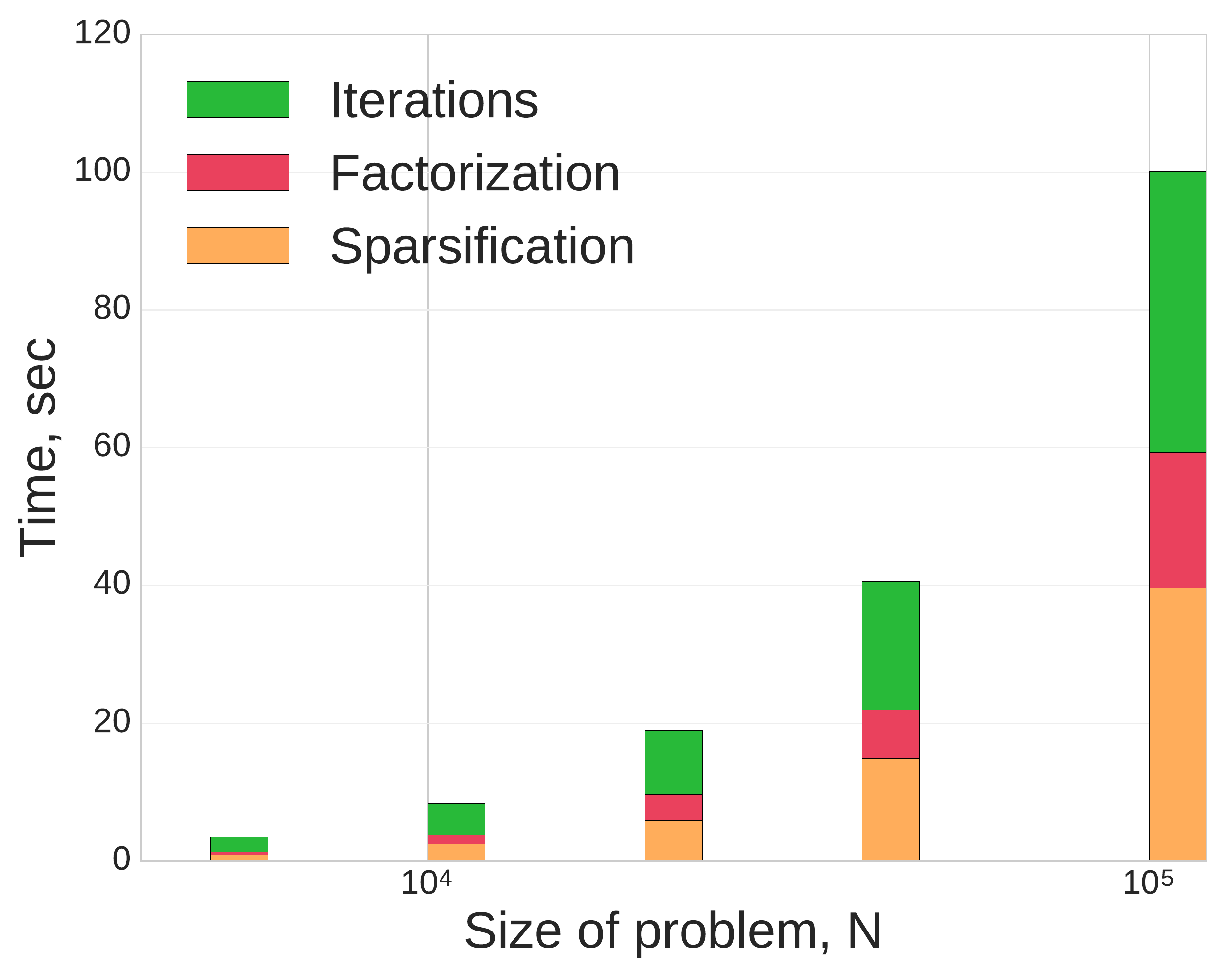}
 \caption{Contribution into total time of sparsification, factorization and iterations}
 \label{fig:it1_tot}
 \end{subfigure}%
 \begin{subfigure}[t]{.5\textwidth}
   \centering
   \includegraphics[scale=0.24]{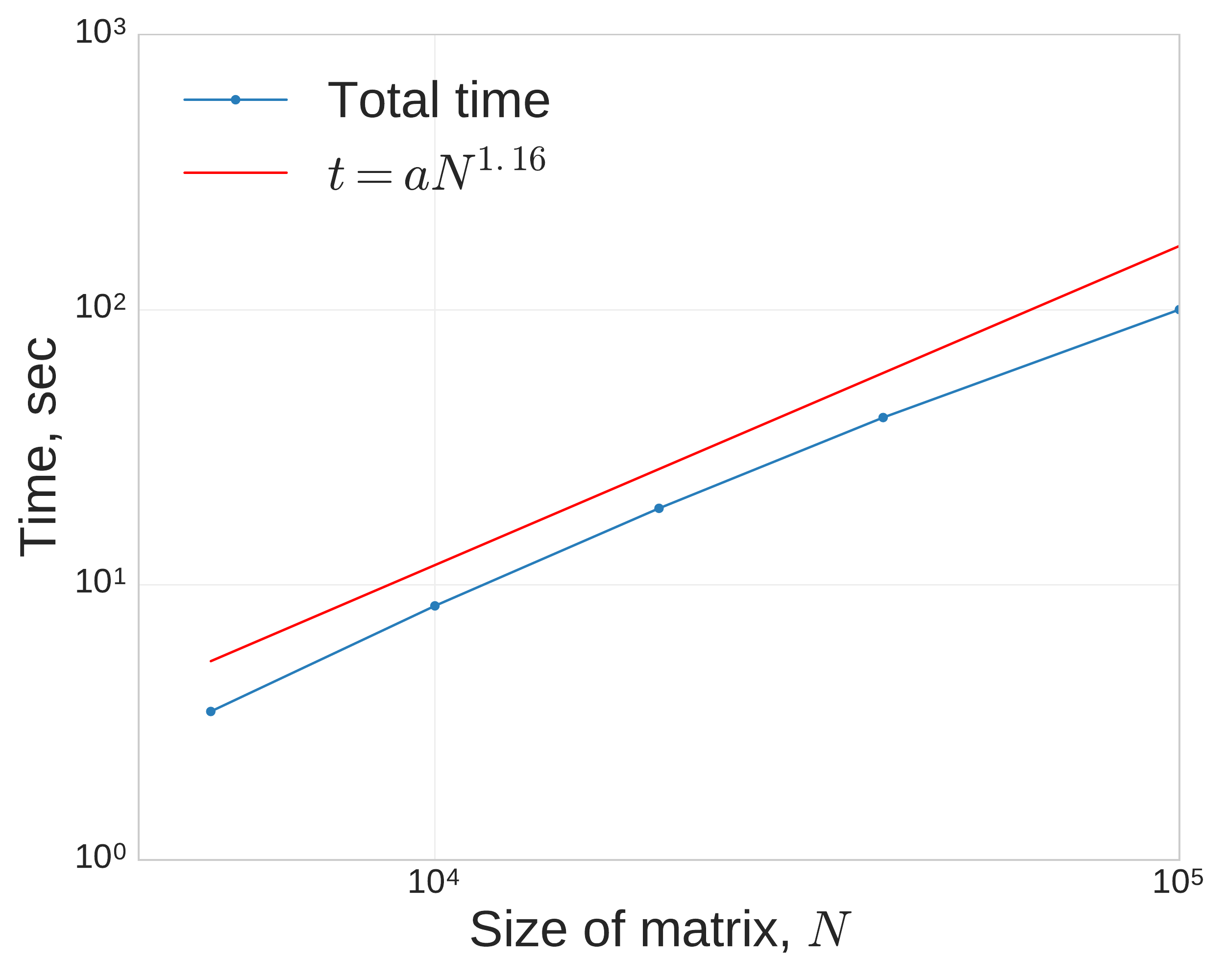}
   \caption{Total time}
   \label{fig:it2_tot}
 \end{subfigure}
 \caption{Total solution time}
\end{figure}%

\textbf{Example with ill-conditioned matrix.}
Consider the matrix from Example~\ref{ex:mat1} with $d = 10^{-2}$, condition number $\mathrm{cond}(A) = 10^4$.
As in the previous paragraph, we solve this system using GMRES iterative solver for the matrix approximated in \HT~format with accuracy $\epsilon = 10^{-9}$ and as a preconditioner we used \HT~approximation of the matrix $A$ with accuracy $\epsilon = 10^{-3}$ sparsified and factorized with ILUt decomposition. In Figure~\ref{fig:it2_comv} convergence  till the tolerance $10^{-10}$ for different ILUt parameters is shown.

\begin{figure}[H]
 \centering
 \includegraphics[scale=0.34]{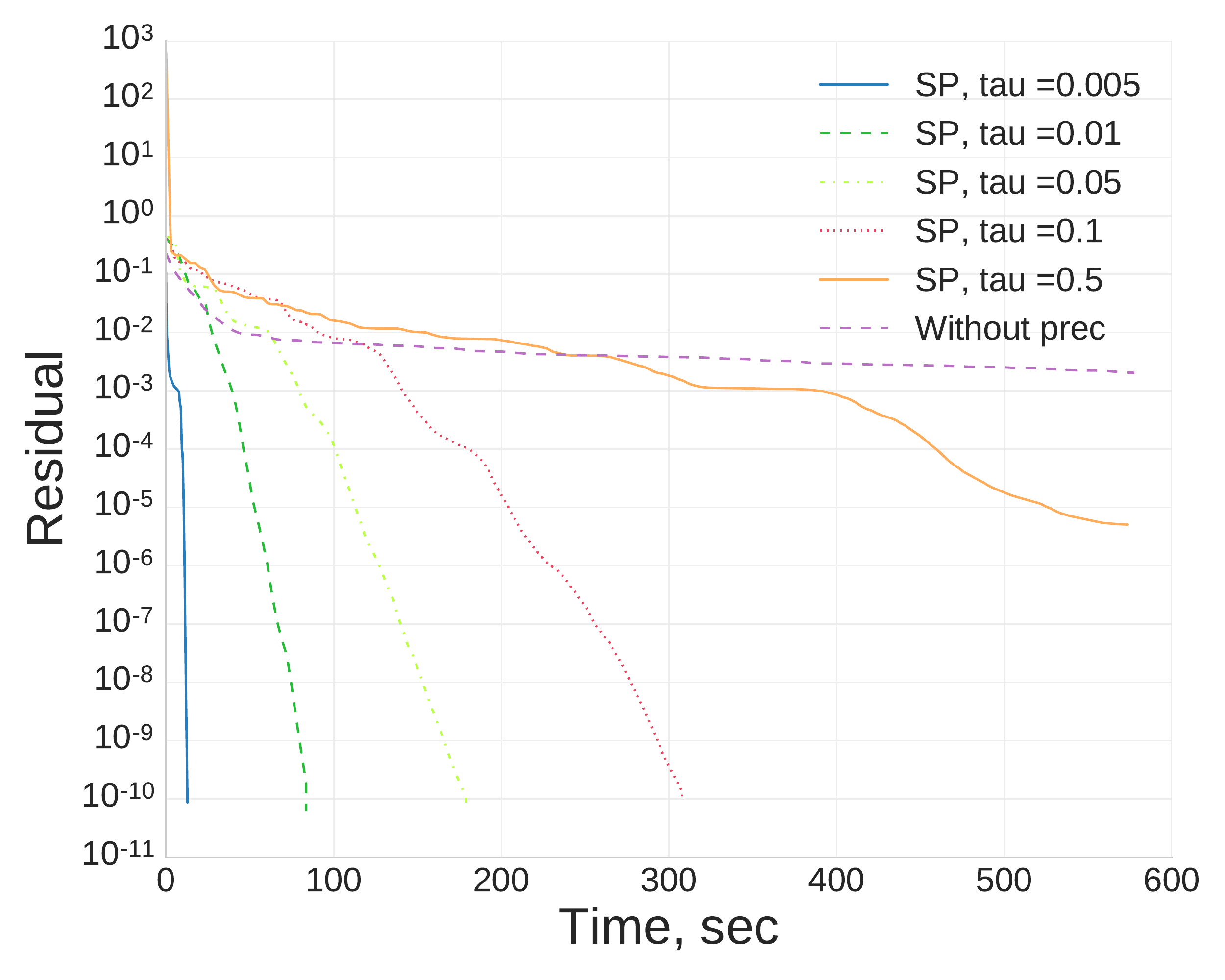}
 \caption{Convergence of GMRES with different preconditioners, $N = 10^5$}
 \label{fig:it2_comv}
\end{figure}%

In Figure~\ref{fig:it2_time} the total solution time for different ILUt parameters is shown.

\begin{figure}[H]
 \centering
 \includegraphics[scale=0.34]{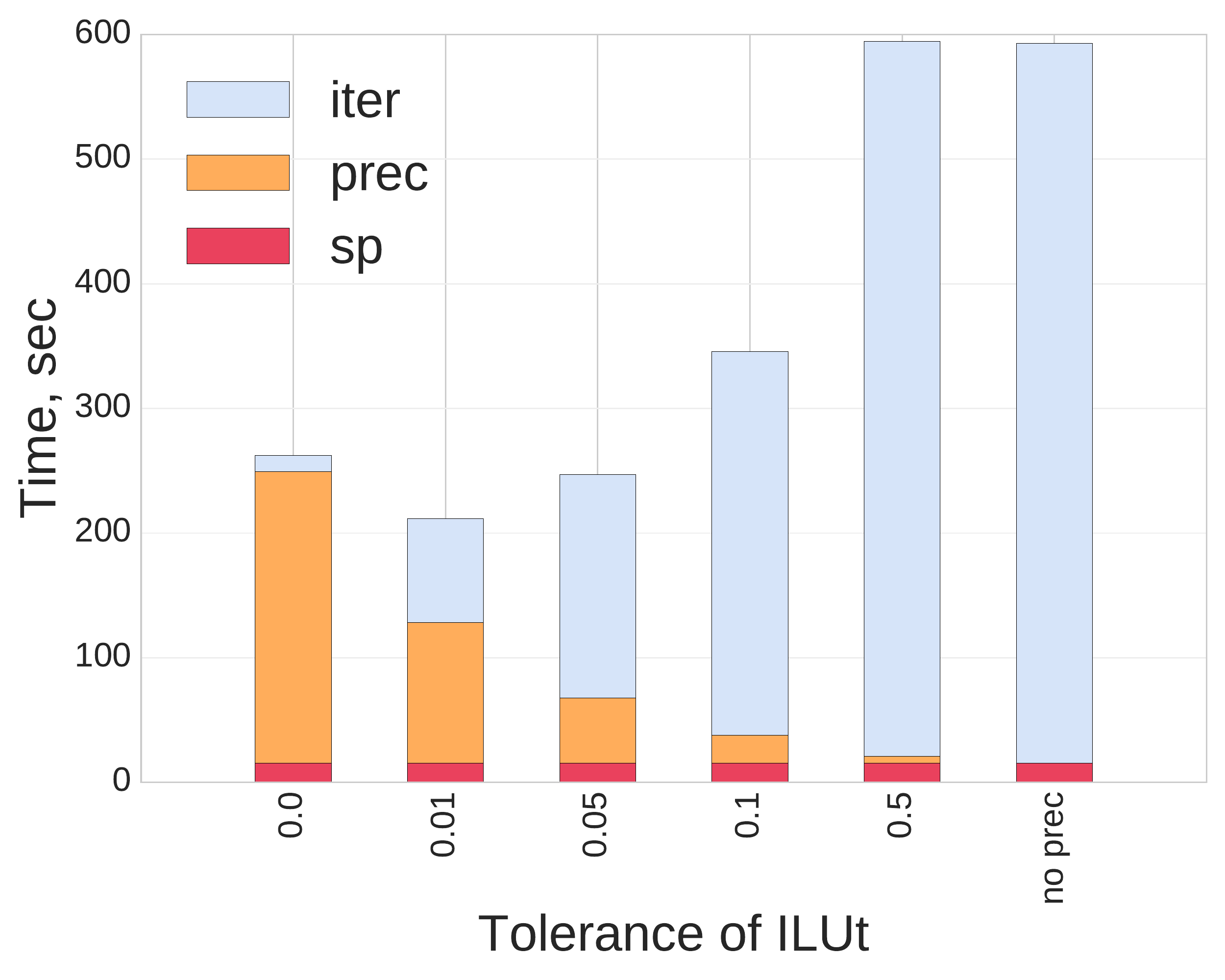}
 \caption{Contribution into total time of sparsification, factorization and iterations, $N = 10^5$}
 \label{fig:it2_time}
\end{figure}%

As we can see, the optimal ILUt parameter for this problem is $\tau = 10^{-2}$.
Figure~\ref{fig:it2_tot} presents the total time required for solution of the system with optimal ILUt parameter.

\begin{figure}[H]
 \centering
 \begin{subfigure}[t]{.5\textwidth}
   \centering
     \includegraphics[scale=0.24]{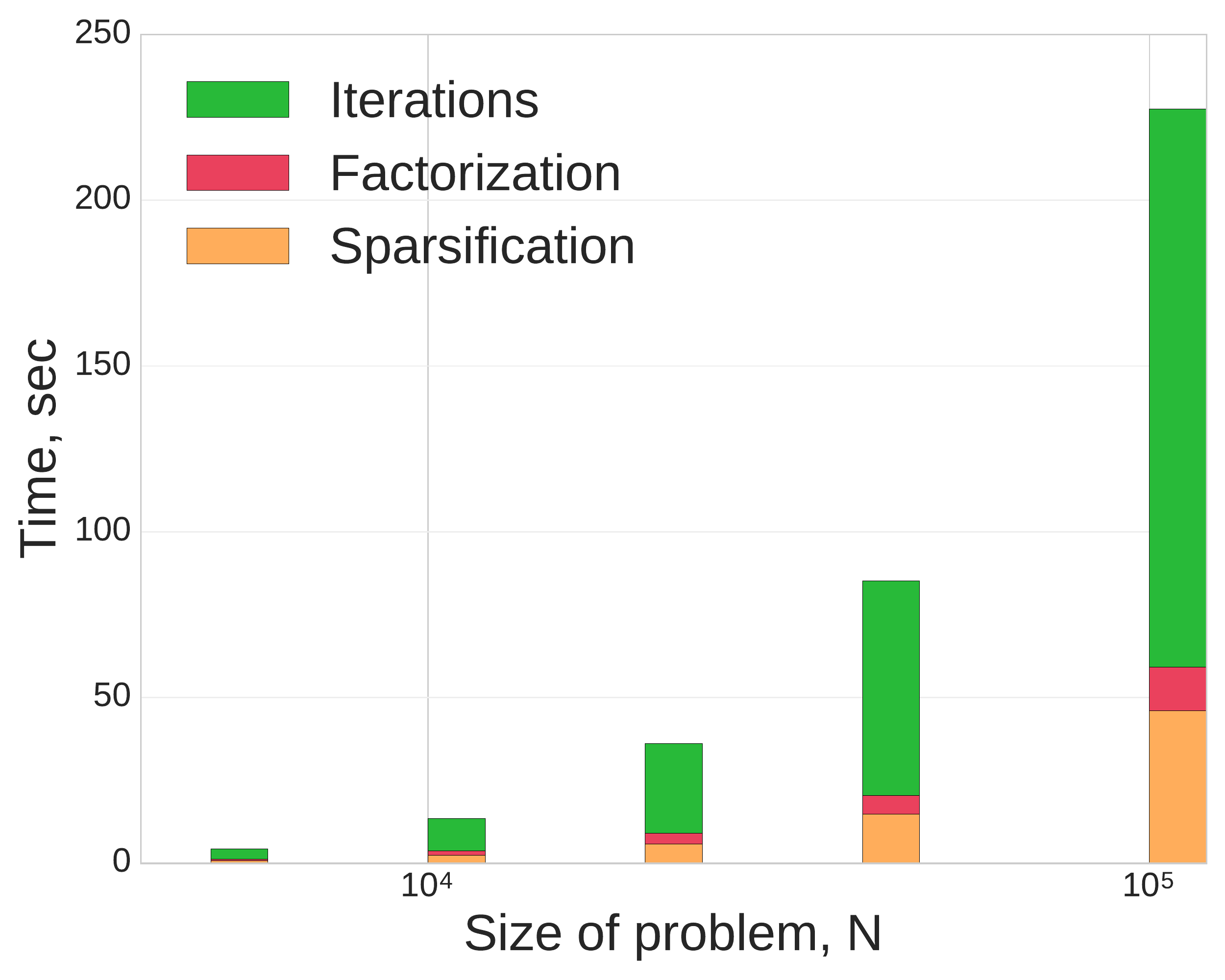}
     \caption{Contribution into total time of sparsification, factorization and iterations}
     \label{fig:it2_tot}
 \end{subfigure}%
 \begin{subfigure}[t]{.5\textwidth}
   \centering
   \includegraphics[scale=0.24]{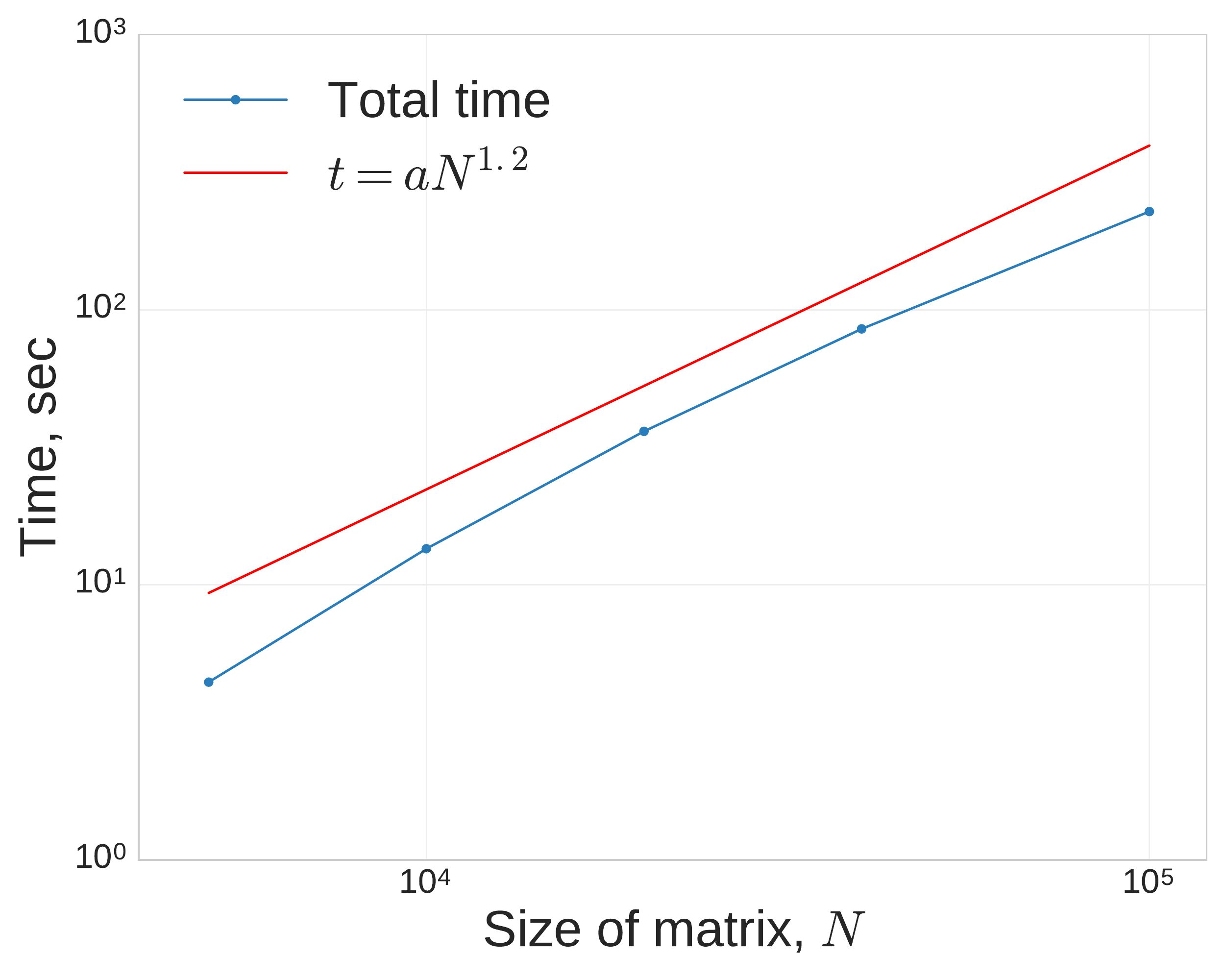}
   \caption{Total time}
   \label{fig:it2_tot}
 \end{subfigure}
 \caption{Total solution time}
\end{figure}%
%
We compared the results of our solver (GMRES preconditioned by sparsification method) with results of IFMM solver presented in\cite{darve-ifmm_prec-2015} for the same problem, same parameters and similar hardware. For test we used matrix from Example~\ref{ex:mat1} with 3D data, matrix size $N = 10^5$, iterative method GMRES till the residual $r = 10^{-10}$, parameters $d = \{ 10^{-3},  10^{-2}\}$ for well- and ill-conditioned systems. We present the best total solution time achieved in experiments for both methods in Table~\ref{tab:ifmm}.
\begin{table}[H]
\centering
\begin{tabular}{ l | c  c  c  c }
 &$d = 10^{-3}$ & $d = 10^{-2}$ \\
\hline
IFMM solution time,~sec & 118& 301\\
Sparsification solution time,~sec&96& 218 \\
\end{tabular}
\caption{Comparison with IFMM method.}
\label{tab:ifmm}
\end{table}
Note that sparsification method is implemented in Python programing language and can be significantly improved by applying Cython or switching to another programing language (as C++ or Fortran).

\section{Related work}
Hierarchical low-rank  matrix formats such as $\mathcal{H}${\color{royalblue}\cite{hackbusch-h-1999,hackbusch-h-2000,hackbusch-h2-2000,Borm-h-2003}},
HODLR\cite{ambikasaran-hodlersolver-2013,greengard-hodlr-2016} {\color{royalblue} (Hierarchical Off-Diagonal Low-Rank)},
HSS\cite{martinsson-hss_integr-2005,chandrasekaran-HSSsolver-2006,ShDewilde-hss-2007} {\color{royalblue} (Hierarchically \\ Semiseparable)},
\HT~\cite{hackbusch-h2-2000,Borm-h2-2010} matrices and etc., that are matrix analogies of the fast multipole
method\cite{GrRo-fmm-1988,GrRo-fmm-1987}, have two significant features: they do store information in data-sparse formats and they provide the fast matrix by vector product. Fast ($\OB(N)$, where $N$ is size of the matrix) matrix by vector product allows to apply iterative solvers. Data-sparse representation allows to store matrix in $\OB(N)$ cells of memory, but storage scheme is usually complicated.

If the hierarchical matrix is ill-conditioned, then pure iterative solver fails and it is required to apply either approximate direct solver or preconditioner (that is also approximate direct solver, probably with lower accuracy).  Due to complex storage schemes of hierarchical matrices, construction of the approximate direct solver is a challenging problem.  There exists two general approaches to approximate direct solution of \HT~matrix: factorization of hierarchical matrix\cite{Beb-hlu-2005,martinsson-hss_integr-2005,ShDewilde-hss-2007}, and sparsification of the hierarchical matrix followed by factorization of the sparse matrix\cite{Ambikasaran-ifmm-2014,sushnikova-se-2016}.

The factorization approach is more popular for hierarchical matrices with strong low-rank structure, also known as hierarchical matrices with weak-admissibility criteria{\color{royalblue}\cite{hackbusch-weak-2004}} \\($\mathcal{H}${\color{royalblue}\cite{hackbusch-h-1999,hackbusch-h-2000}},
HODLR{\color{royalblue}\cite{ambikasaran-hodlersolver-2013,greengard-hodlr-2016}},
HSS{\color{royalblue}\cite{martinsson-hss_integr-2005,chandrasekaran-HSSsolver-2006,ShDewilde-hss-2007}} matrices).
{\color{royalblue} For the $\mathcal{H}$ matrix, the algorithm $\mathcal{H}$-LU\cite{Beb-hlu-2005} with almost linear complexity was proposed. This algorithm has been successfully applied to many problems.} The major drawback of the $\mathcal{H}$-LU algorithm is that factorization time and memory required for $L$ and $U$ factors can be quite large.
Approximate direct solvers based on factorization of
HSS and
HODLR matrices are {\color{royalblue}also} well studied {\color{royalblue} and found
many\cite{ShDewilde-hss-2007,XiaSh-hss-2009,Mar-hss-2011,solovyev-hss-2014,
ambikasaran-hodlersolver-2013,greengard-hodlr-2016,martinsson-hss_integr-2005}
successful applications}.

{\color{royalblue} One of the approaches to the solution of the systems with \HT~matrices is the sparse factorization (sparsification).} The sparsification approach is usually applied to hierarchical matrices with weak-admissibility criteria (\HT~matrices). Sparsification algorithms transform \HT~matrix into the sparse matrix and then factorize the sparse matrix.
Algorithm, proposed in this paper is the sparsification algorithm. The main difference between presented work and the other {\color{royalblue} sparse factorizations\cite{Ambikasaran-ifmm-2014,darve-ifmm_prec-2015,sushnikova-se-2016} is a size of sparse factors. The main benefit of the non-extensive sparsification is that it preserves the size of the factorized \HT~matrix, while the other sparsification algorithms return extended factors.
}

\section{Conclusions}
We have proposed a new approach to the solution of the systems with \HT~matrices that is based on sufficient non-extensive sparsification of the \HT~matrix. Proposed sparsification is suitable for any \HT~matrices (including non-symmetric) and preserves such important properties of the matrix as its size, symmetry (if exists) and positive definite (if exists).

\bibliography{lib}
\end{document}